\newcommand{\les}{\lesssim}

\def\Qr{\mbox{Qr}}

\def\Div{{\bf Div}}

\def\Bund{\underline{B}}
\def\kh{\hat{k}}

\documentstyle[amssymb,amsfonts]{amsart}

\newenvironment{proof}{\noindent {\bf Proof} }{\endprf\par}
\def \endprf{\hfill  {\vrule height6pt width6pt depth0pt}\medskip}
\def\emph#1{{\it #1}}
\def\textbf#1{{\bf #1}}
\newcommand{\bea}{\begin{eqnarray}}
\newcommand{\eea}{\end{eqnarray}}
\def\beaa{\begin{eqnarray*}}
\def\eeaa{\end{eqnarray*}}

\def\ba{\begin{array}}
\def\ea{\end{array}}
\def\be#1{\begin{equation} \label{#1}}
\def \eeq{\end{equation}}

\newcommand{\nn}{\nonumber}

\def\a{{\alpha}}
\def\b{{\beta}}
\def\ga{\gamma}
\def\Ga{\Gamma}
\def\de{\delta}
\def\De{\Delta}
\def\ep{\epsilon}

\def\la{\lambda}

\def\si{\sigma}
\def\Si{\Sigma}
\def\om{\omega}

\def\ze{\zeta}
\def\nab{\nabla}

\def\vphi{\varphi}
\def\aa{{\underline{\a}}}
\def\bb{{\underline{\b}}}
\def\bb{{\underline{\b}}}

\def\Lb{{\underline{L}}}

\newcommand{\trchb}{\tr \chib}
\newcommand{\chih}{\hat{\chi}}
\newcommand{\chib}{\underline{\chi}}

\newcommand{\etab}{\underline{\eta}}

\def\chih{\hat{\chi}}
\def\trch{\mbox{tr}\chi}
\def\tr{\mbox{tr}}
\def\Tr{\mbox{Tr}}
\def\Xb{{\underline X}}
\def\Yb{{\underline Y}}

\def\Null{\dot{\NN}^{-}}

\def\MM{{\cal M}}
\def\NN{{\cal N}}
\def\II{{\cal I}}
\def\FF{{\cal F}}
\def\EE{{\cal E}}

\def\JJ{{\cal J}}

\def\Lie{{\cal L}}

\def\RR{{\cal R}}

\def\GG{{\cal G}}

\def\TT{{\cal T}}

\def\A{{\bf A}}
\def\B{{\bf B}}
\def\D{{\bf D}}
\def\F{{\bf F}}
\def\G{{\bf G}}

\def\J{{\bf J}}
\def\M{{\bf M}}
\def\L{{\bf L}}

\def\Q{{\bf Q}}
\def\R{{\bf R}}

\def\U{{\bf U}}

\def\T{{\bf T}}
\def\g{{\bf g}}
\def\m{{\bf m}}
\def\t{{\bf t}}

\def\RRR{{\Bbb R}}

\def\SSS{{\Bbb S}}

\def\piT{{\,^{(\T)}\pi}}

\def\na{\overline{\nabb}}
\newcommand{\nabb}{\mbox{$\nabla \mkern-13mu /$\,}}

\def\B{{\bf B}}

\def\lap{\Delta}
\def\pr{\partial}

\def\c{\cdot}

\renewcommand{\div}{\mbox{div }}
\newcommand{\curl}{\mbox{curl }}

\def\err{\mbox{Err}}
\newcommand{\lapp}{\mbox{$\bigtriangleup  \mkern-13mu / \,$}}
\def\f14{{\frac{1}{4}}}
\def\f12{{\frac{1}{2}}}

\newcommand{\piX}{\,^{(X)}\pi}

\def\dual{{\,^\star\, \mkern-3mu}}
\def\2{{\overline 2}}

\begin{document}
\theoremstyle{plain}
  \newtheorem{theorem}[subsection]{Theorem}
  \newtheorem{conjecture}[subsection]{Conjecture}
  \newtheorem{proposition}[subsection]{Proposition}
  \newtheorem{lemma}[subsection]{Lemma}
  \newtheorem{corollary}[subsection]{Corollary}

\theoremstyle{remark}
  \newtheorem{remark}[subsection]{Remark}
  \newtheorem{remarks}[subsection]{Remarks}

\theoremstyle{definition}
  \newtheorem{definition}[subsection]{Definition}

\include{psfig}

\include{psfig}
\title[Curvature ]{On the breakdown criterion in General Relativity }
\author{Sergiu Klainerman}
\address{Department of Mathematics, Princeton University,
 Princeton NJ 08544}
\email{ seri@@math.princeton.edu}

\author{Igor Rodnianski}
\address{Department of Mathematics, Princeton University, 
Princeton NJ 08544}
\email{ irod@@math.princeton.edu}
\subjclass{35J10\newline\newline
The first author is partially supported by NSF grant 
DMS-0070696. The second author is partially 
supported by NSF grant DMS-0702270.
}
\begin{abstract}
We give a  geometric criterion for the breakdown of an Einstein vacuum  space-time foliated by a constant mean curvature, or maximal, foliation. 
More precisely we show that the foliated space-time can be extended as long as the the second fundamental form and  the first  derivatives of the logarithm of the  lapse of the foliation remain uniformly bounded. We make no  restrictions on the size of the initial data.
\end{abstract}
\maketitle
\vspace{-0.3in}
\section{Introduction}

This paper is concerned with the problem of a {\it geometric} criterion 
for breakdown of solutions $(\M,\g)$ of the vacuum Einstein equations. 
\be{eq:Einst-vacuum}
\R_{\a\b} (\g)=0.
\end{equation}
To describe the problem we assume that a  part of space-time $\MM_*\subset \M$ is foliated 
by the level hypersurface of a time function $t$, monotonically
increasing towards future,  with lapse $n$ and 
second fundamental form $k$ defined by,
\be{eq:def-k-n}
k(X,Y)=-\g(\D_X \T,Y),\quad \qquad  n=\big(-\g(\D t, \D
t)\big)^{-1/2}
\end{equation}
where  $\T$ is the future unit normal to $\Si_t$, $\D$ 
 is the space-time  covariant derivative associated with $\g$, and $ X,Y $ are tangent to $\Si_t$
and . Let $\Si_0$ be a fixed leaf  of the $t$ foliation,
 corresponding to $t=t_0$.  We shall refer to $\Si_0$
as initial slice. 
We assume
that the space-time region $\MM_*$ is globally hyperbolic, i.e. every
 causal curve from a point $p\in \MM_*$  intersects  $\Si_0$ at
 precisely one point. We  also  assume that the initial slice 
verifies the following assumption.

{\bf A 1.}\quad 
There exists a finite  covering 
of $\Si_0$ by a finite number of   charts $U$ such
that for any fixed chart, the induced metric $g$  verifies
\be{eq:assum-2}
 \De_0^{-1}|\xi|^2\le g_{ij}(x) \xi_i\xi_j \le \De_0|\xi|^2, \qquad \forall
x\in U
\end{equation}
with $\De_0$ a fixed positive number.

  We consider the
 following two situations:
\begin{enumerate}
\item
The surfaces $\Si_t$ are asymptotically flat and  maximal. 
$$\tr k = 0.$$

\item
The surfaces $\Si_t$ are compact,  of Yamabe type $-1$, and of 
 constant, negative mean curvature.   They form 
what is called   a (CMC) foliation .
 $$tr k=t, \qquad t<0$$
\end{enumerate}

Though our methods apply equally well to both situations
we shall only consider here the 
    latter case,  which is somewhat easier to treat due to the
compactness of the level surfaces $\Si_t$. We shall thus assume in what follows that
the region $\MM_*$   is equal to $\cup_{t\in[t_0, t_*)}\Si_t$, with $t_*<0$.
We can also assume that the initial hypersurface $\Si_0$ corresponds to  $t_0= -1$.

{\bf Remark}.\quad 
In the second case the CMC conjecture asserts that it should be possible to extend 
the foliation, in a smooth manner, to 
its maximal allowed value $\tr k=t=0$, see \cite{And} and references therein.

Given $p\in \MM_*$
 we can define a point-wise norm $|\Pi(p)|$ of any 
space-time tensor $\Pi$ via decomposition 
$$
X=- X^0 \T + \underline X,\quad X\in T{\cal M},\quad \underline X\in T\Si_t
$$
  We denote by $\|\Pi(t)\|_{L^p}$ the $L^p$ norm of $\Pi$
on $\Si_t$. More precisely,
\beaa
\|\Pi(t)\|_{L^p}=\int_{\Si_t}|\Pi|^p dv_g
\eeaa
with $dv_g$ the volume element of the metric $g$ of $\Si_t$.
The main result of this paper is the following theorem.

\begin{theorem}[Main theorem]
\label{thm:main}
Let $(\M,\g)$ be a globally hyperbolic development of $\Si_{0}$ foliated  by the CMC
level hypersurfaces of a time function
$t<0$, such that  $\Si_0$  corresponds 
to   the level surface $ t=t_0$.  Assume that $\Si_0$ verifies  {\bf A1}. Then the first time $T_*<0$, with respect to
 the $t$-foliation,
  of a breakdown is characterized by the condition 
\be{eq:criterion}
\lim\sup_{t\to T_*^-} \big (\, \|k(t)\|_{L^\infty}+ \|\nab \log n(t)\|_{L^\infty}\big )=\infty
\end{equation}
More precisely the space-time together with the foliation $\Si_t$ can be extended
beyond any value $t_*<0$ for which,
\be{eq:criterion-finite}
 \sup_{t\in[t_0, t_*)}\|k(t)\|_{L^\infty}+ \|\nab \log n(t)\|_{L^\infty}=\De_0<\infty
\end{equation}
\end{theorem}
Condition \eqref{eq:criterion-finite} can be reformulated in terms
of the deformation tensor  of the  future unit normal $\T$, $\pi=\piT=\Lie_\T\g$. 
By a simple calculation, expressed relative to an orthonormal frame 
$e_0=\T, e_1,e_2, e_3$, we find,
\be{eq:comp-pi}
\pi_{00}=0,\quad\pi_{0i}=n^{-1} \nab_i n,\quad \pi_{ij}=-2k_{ij}.
\end{equation}
Consistent with the statement of the main theorem we assume that $\T$
is an approximate Killing vectorfield in the following sense,

{\bf A2}.\quad There exists a constant $\De_0$ such that,
\be{eq:criterion-finite-pi}
\sup_{t\in [t_0, t_*)}\|\pi(t)\|_{L^\infty}\le \De_0
\end{equation} 
In addition to the constant $\De_0$ in {\bf A1}, {\bf A2} we introduce another
constant $\RR_0$ which plays an important role in the proof, which bounds the $L^2$ norm
of the spacetime curvature tensor $\R$ on $\Si_0$,

\be{eq:const-R0}
\|\R(t_0)\|_{L^2(\Si_0)}\le \RR_0
\end{equation}

To prove our main theorem we have to  show that if assumptions {\bf A1} and {\bf A2}
 are satisfied then  
the space-time $\MM_*$
 can be extended beyond $t_*$. We want to emphasize that theorem \ref{thm:main}
is a large data resulty; indeed we make no smallness assumptions on the constants
$\De_0$ and $\RR_0$.

Our theorem is connected and partially  motivated by the following three  earlier
 breakdown criteria  results:

 {\bf  1.}\quad  The first is a result of M. Andersson, \cite{And},  who showed that a
breakdown can
 be tied to the condition that 
 $$
 \lim\sup_{t\to t_*^-} \, \|\R(t)\|_{L^\infty}=\infty.
 $$
 Our result can be viewed as complimentary. It is clear however that 
 the condition \eqref{eq:criterion} is formally weaker as it refers only to 
 the second fundamental form $k$ and the lapse $n$ which requires one degree
 less of differentiability. Moreover a condition on the  boundedness
 of the $L^\infty$ norm of $\R$ 
exhausts  all the  dynamical degrees of 
freedom of the equations. Indeed, once we know that  $\|\R(t)\|_{L^\infty}$
is finite, one can  find bounds for  $n$, $\nab n$  and $k$
on $\Si_t$ purely by elliptic estimates. This is certainly  not true in our   case.
 
 \vskip 1pc

 \vskip 1pc
{\bf 2.}\quad  Our result can be also compared to  the well known  Beale-Kato-Majda,
\cite{BKM},
 criterion for  breakdown of solutions of the incompressible Euler equation
 $$
\pr_t  v + (v\c \nab ) v =-\nab p,\qquad 
\div v =0,
 $$
with smooth initial data at $t=t_0$.
A routine application of the energy estimates shows that solution $v$ blows up 
if and only if 
\be{eq:BKM1}
\int_{t_0}^{t_*} \|\nab v(t)\|_{L^\infty} dt =\infty.
\end{equation}
 The Beale-Kato-Majda improves the blow up criterion by replacing it with the 
 following  condition on  the vorticity $\omega=\curl v$:
 \be{eq:BKM2}
\int_{t_0}^{t_*} \|\omega (t)\|_{L^\infty} dt =\infty.
 \end{equation}
 To relate $\nab v$ and $\om$  one observes that  
$$
\div v =0, \quad \curl v=\omega
$$
 forms an elliptic system for $v$  in terms of $\om$. Thus $\nab v$ can be expressed in
terms  $\omega$ via a singular integral operator, i.e.  
a zero order pseudodifferential operator:
 \be{eq:BKM3}
 \nab v = P^0 (\omega).
 \end{equation}
 Although $P^0$ does not define a bounded map $L^\infty\to L^\infty$ it can 
 be shown that \eqref{eq:BKM3} is sufficient to reduce 
 the breakdown condition \eqref{eq:BKM1} to the more  satisfying one \eqref{eq:BKM2},
 in terms of the vorticity alone.
 
 Similarly, in the case of the Einstein equations energy estimates,
expressed relative to  a special system of coordinates ( such as wave coordinates),
 show 
 that breakdown does not occur unless 
 $$
 \int_{t_0}^{t_*} \|\pr\g(t)\|_{L^\infty} dt =\infty.
 $$
 This condition however is not geometric as it depends on the choice of a full
 coordinate  system. Observe that  both the spatial derivatives of the lapse $\nab n$
 and the components of the second fundamental form,  $
 k_{ij}=-\f 12 n^{-1}\, \pr_t g_{ij},  $ can be viewed   
 as components  of $\pr\g$. 

 Note however that after prescribing $k$ and $\nab n$ 
we are still left with many  more  degrees of freedom 
in determining $\pr \g$. The fundamental difficulty
that one needs to overcome is that of deriving 
bounds for $\R$ using only bounds for  $\|\nab\log  n(t)\|_{L^\infty}+\|k(t)\|_{L^\infty}$
and geometric informations   on the initial hypersurface $\Si_0$. Clearly
this cannot be done by elliptic estimates alone.
 Thus, as opposed to  both  the results of M. Anderson and  Beale-Kato-Majda,
it is far less obvious that a condition such as 
\eqref{eq:criterion} can cover all 
 {\it dynamic} degrees of freedom of the Einstein equations.
Despite the formal similarity with  
 the previous results  mentioned above,  the proof
of Theorem \ref{thm:main}
 requires a  conceptually different treatment.

{\bf 3.}\quad Finally, the result whose proof  is closest in spirit
 to  ours
and which has  played the main motivating role 
in  developing our approach, 
is the proof of global regularity 
 of solutions of the Yang-Mills equations in ${\Bbb R}^{3+1}$ by Eardley and Moncrief,
 see \cite{EM1}, \cite{EM2}.
 To explain the connection of their  result to  ours
 we review  below its main ideas.

Recall that the curvature tensor $\F_{\a\b}$ of a Yang Mills connection
 $\A_\a dx^\a$, with values in the Lie algebra $su(N)$ is a critical point 
 of the Yang-Mills functional 
 $$
 YM[\F]=\int_{{\Bbb R}^{3+1}} \Tr \,\big (\dual \F\wedge \F\,\big ) 
 $$
and verifies the wave equation,
\be{eq:YM1}
\square_{(\la)} \F= \F\star \F,
\end{equation}
where $\square_{(\la)}$ denotes the covariant
wave operator,
 \beaa
\square_{(\la)}\F&=&\D^\a \D_\a \F=\square \F+[\la, \pr \F]+[\pr \la, \F]+[\la,[\la,\F]],
\eeaa
 $\square $  denotes the usual D'Alembertian in $\RRR^{3+1}$ and $\D_a =\pr_\a+[\la_\a, \c]$
  the gauge  covariant derivative.
Since the Minkowski space-time  metric 
$$
\m=-dt^2 + \de_{ij} dx^i dx^j
$$
is static (in particular $n=1$ and $k=0$) the energy of $\F$ associated with 
the energy-momentum tensor $\Q[\F]_{\a\b}=\F_{\a}^{\,\, \la} \F_{\b\la}+\dual \F_{\a}^{\,\,
\la}\dual \F_{\b\la} $  and vectorfield $\T=\pr_t$ is conserved. 
In particular, the flux of energy $\FF_p$
through the null boundary  $\NN^{-}(p)$ of the domain of dependence $\JJ^{-}(p)$
of an arbitrary  point $p$  can be bounded by the energy of the initial
data which we denote by $I_0$.  We assume that smooth data 
for $F$ is prescribed at $t=0$
and restrict $\JJ^{-}(p)$ and $\NN^{-}(p)$ to $t\ge 0$.  We
recall that the flux has the form,
$
 \FF_p=  \big(\int_{\NN^{-}(p)} \Q[\F](L,\T)\,\big)^{1/2}
$
with 
$L=-\pr_t+\pr_r$ the null
 geodesic generator of $\NN^{-}(p)$ normalized by  the condition $<L,\T>=1$.

The proof of the global regularity of solutions of the Yang-Mills equations 
is based on  the boundedness   of the flux  $\FF_p\le I_0< \infty$. Here are
a summary of the main steps.

{\bf 1.}\quad   Rewrite \eqref{eq:YM1} in the form
$\square \F=\F\star \F-(\square_{(\la)}-\square)\F.$
 Using  the explicit representation, in $\RRR^{3+1}$,  of solutions
to  the inhomogeneous wave equation,
we deduce, for all points $p$, with $t>0$,
\bea
\F(p)&=& (4\pi)^{-1}\int_{\NN^{-}(p;\de)} r^{-1} \F \star \F+\F^{(0)}(p;\de)
\label{eq:YM-form}\\ & -&
(4\pi)^{-1}\int_{\NN^{-}(p;\de)} r^{-1}(\square_{(\la)}-\square)\F \nn .
\eea
Here $\NN^{-}(p,\de)$ represents the portion
of the null cone $\NN^{-}(p)$ included in the time slab $[t(p)-\de, t(p))$,
with $t(p)$ the value of  the time parameter at $p$. Also   $r$ is the distance,
in euclidean sense, to the vertex $p$ and $\F^{(0)}(p;\de)$ represents
an  homogeneous solution to  the wave equation whose initial data at $t=t(p)-\de$
coincide with  those of  $\F$.  

{\bf 2.}\quad Ignore, for a moment,
the presence of the third term on the right hand side of 
\eqref{eq:YM-form}. Using the explicit 
form of the nonlinear term $\F\star \F$ one notices that at least one
component of the product can be estimated by the flux  $\FF_p$ of $\F$ through
the null hypersurface $\NN^{-}(p)$. Denoting  $|\F|=\sum_{\a\b} |\F_{\a\b}|$, we have by  a
simple estimate,
\beaa
\qquad\qquad\qquad|(\F(p)-\F^{(0)}(p;\de)|&\les&  \FF_p
\big(\int_{\NN^{-}(p;\de)}
r^{-2}\big)^{1/2} \,\,\|\F\|_{L^\infty(\JJ^{-}(p;\de))}\\
&\les& \de^{1/2} \FF_p
\|\F\|_{L^\infty(\JJ^{-}(p,\de))}
\eeaa
   where $ \|\F\|_{L^\infty(\JJ^{-}(p,\de))}$ denotes the sup- norm of $|\F|$
for all points in the domain of dependence  $\JJ^{-}(p)$ of $p$
intersected with  the slab $t(p)-\de, t(p)]$. Therefore, we deduce the following,
\begin{lemma}
\label{le:YM2}
 If $\,\, \de^{1/2}\c \FF_p$ is sufficiently small, then for any $t\ge 0$
\be{eq:YM2}
\|\F(t)\|_{L^\infty}\les
 \|\F(t-\de)\|_{L^\infty} +\|\D \F(t-\de)\|_{L^\infty}
\end{equation}
\end{lemma}
{\bf 3.}\quad 
   Arguing recursively and using the standard local existence theorem
 for the Yang-Mills system\footnote{In a given gauge such as   the Coulomb gauge}, 
 one can find    bounds for all components of the curvature tensor\footnote{Once bounds are
established for  $F(p)$ one can proceed in the same
manner to derive bounds for derivatives of $\F$ at $p$.} 
$\F(p)$ 
 depending only on the fact that  $\FF_p$ is  uniformly bounded
 and the initial data
 data $\F(0)$ is
smooth.

{\bf 4.}\quad One can show that \eqref{eq:YM2}
remains true even as we take into consideration the presence
of the third term in \eqref{eq:YM-form}. Consider for example
 what could be, potentially, the most dangerous term,
\beaa
\int_{\NN^{-}(p;\de)} r^{-1} \la \c \pr \F.
\eeaa
Here we have to hope that we can integrate by parts
to transfer the derivative from $\F$ to $\la $. This can only be done if 
$\la\c\pr $ is tangential to the light cone $\NN^{-}(p)$.  Miraculously,  this 
 can be achieved  by  taking $\la$ in  
 the Cronstr\"om gauge, that is
one assumes that the conection 1-form $\la$  satisfies
\be{eq:YM3}
 (x-y)^\a \la_\a=0,
\end{equation}
where $x^\a$ are the space-time coordinates of $p$ 
and $y^\a$ those of a point $q\in \NN^-(p)$.
With this choice, after the integration by parts, one can
 treat all  the remaining terms
 in $(\square_{(\la)}-\square ) \F$ roughly  in the same way 
as the main term $F\star F$. To  show  this one  has to observe that 
the value of $\la$ at any point $q$ in the domain of dependence
of $p$ can be estimated by $\|\F\|_{L^\infty(\JJ^{-}(p,\de))}$. This leads
 to the same estimate \eqref{eq:YM2} as stated in the Lemma above.
 
{\bf 5.}\quad In \cite{Kl-Ma}  the global  regularity result was reproved by 
strenghening the classical local existence result to $\la\in H^1(\RRR^3)$
and $E\in L^2(\RRR^3)$, which is at the  same regularity level
as  the energy norm. That required, instead of the pointwise
estimates
\eqref{eq:YM2}, a new generation of $L^4$ type estimates, called bilinear.
The premise of the \cite{Kl-Ma} approach was  the fact that,  once we have
 a local existence result which depends only on the energy norm
of the  initial data, global existence can be easily derived 
by a simple continuation argument.

{\bf 6.}\quad In  \cite{Kirch} we have developed a gauge independent  approach to    the proof of the  Eardley-Moncrief result.  The approach is based on a Kirchoff-Sobolev parametrix for $\square_{(\la)}$, similar to the one we use
in this paper,   which replaces  \eqref{eq:YM-form} by a  
gauge invariant  formula  depending, implicitly\footnote{Through transport equations along the  null boundary of the causal past of $p$},  
only on  the values  of   $\F$  along  $\NN^{-}(p)$.

This  paper was motivated in part by the desire to adapt the Eardley-Moncrief argument\footnote{Adapting \cite{Kl-Ma}
to General Relativity is the goal of the bounded $L^2$-curvature conjecture, see \cite{kl}.}
to General Relativity. The above discussion indicates that the Eardley-Moncrief proof 
relies on two independent ingredients: conservation of energy and pointwise bounds on
curvature, which depend only on the flux  and  initial data. Since the analogue
of the Yang-Mills energy in General Relativity (the Bel-Robinson energy) is not conserved
one can only hope to reproduce the second part of the Eardley-Moncrief argument
 and prove a conditional regularity result
which states, roughly, that smooth solutions of the Einstein
 equations, in vaccum,  remain smooth, and can therefore  be continued, as long as 
an integral quantity, we call the flux of  curvature, remains bounded.
 A  possibility of such a result  first became 
apparent  to us  in a    discussion 
with  V.  Moncrief\footnote{V. Moncrief  has been  independently  pursuing the analogy between the Einstein and Yang-Mills equations  by developing  an integral 
representation of the curvature tensor in General Relativity
based on the Hadamard-Friedlander method (as in 
\cite{Fried}), see  \cite{M}.   }. Such a  result could also be deduced, in principle,
 from the stronger bounded $L^2$-curvature conjecture, 
accoding to which the initial value problem 
is well posed for initial data sets with $L^2$
bounds on  its  curvature. 
In this paper we actually take a step closer to implementing the full analogue of the
Eardley-Moncrief result. Rather then imposing a direct condition on the finiteness of
the Bel-Robinson energy and curvature flux we formulate   conditions (perhaps more
natural albeit more restrictive) which  control the extent to which the energy is 
not conserved. These conditions, which form our breakdown criterion, involve uniform  bounds
on the second fundamental form $k$ and  derivatives of the lapse $n$.

 In what follows we give a short summary 
of how the mains ideas in the proof of 
the Eardley-Moncrief result for Yang-Mills can be adapted to GR.

{\bf 1.}\quad  A  The curvature tensor $\R$
  of a  $3+1$ dimensional vacum spacetime $(\M,\g)$, see \eqref{eq:Einst-vacuum},
verifies a wave equation of the form,
\be{eq:GR1}
\square_\g \R=\R\star\R
\end{equation} 
where  $\square_\g $ denotes
the covariant wave operator $\square_\g =\D^\a\D_\a$.

{\bf 2.}\quad
The  Bel-Robinson energy-momentum tensor has the form 
$$
\Q[\R]_{\a\b\ga\de}=\R_{\a\la\ga\mu}\R_{\b\,\,\de}^{\,\la\,\,\mu}+
\dual\, \R_{\a\la\ga\mu}\dual\,\R_{\b\,\,\de}^{\,\la\,\,\mu}. 
$$
and verifies,
$\D^\de \Q_{\a\b\ga\de}=0.$ It can thus be used to derive energy
and flux estimates for thee curvature tensor $\R$.
As opposed to the case of the Yang-Mills theory,  however,
in General Relativity the background metric is a  dynamic  variable itself  and thus 
does not admit, in general,  Killing fields (and in particular 
a time-like Killing field). 
This means that we can not associate conserved quantities to a divergence
free Bel-Robinson tensor. It is at this point where we need crucially
our  approximate Killing condition  {\bf A2}. Indeed that condition suffices
to  derive bounds for both   energy and  flux associated to the curvature
tensor $\R$. Using the Bel-Robinson energy momentum tensor  $\Q$ the energy associated
to a slice $\Si_t$ is defined by the integral 
\bea
\EE(t)=\int_{\Si_{t_2}}  \Q[\R] (\T,\T,\T,\T)
\eea
while the flux, through the null boundary   $\NN^{-}(p)$ of the domain of
dependence (or causal past)
$\JJ^{-}(p)$ of a point $p$, is given by  the integral
\be{eq:GR-flux}
 \FF^{-}(p)=\big (\int_{\NN^{-}(p)} \Q[\R](L,\T, \T, \T)\big )^{\frac 12}
\end{equation}
where $L$
is the null geodesic generator of $\NN^{-}(p)$ normalized at the vertex $p$
by $<L,\T>=1$.

As in the case of the Yang-Mills equations it is precisely the boundedness of the 
flux of curvature that plays a crucial role in our analysis. In General Relativity 
the flux takes on even more fundamental role as it is also needed  to control 
the geometry of the very object it is defined on, i.e.  the boundary of the causal 
past of $p$. This boundary,  unlike  in  the case of  
  Minkowski space, is not determined a-priori but
 depends in fact on the  space-time we are trying to control.

{\bf 3.}\quad
 In the construction of  a parametrix for \eqref{eq:GR1}
we cannot, in any meaningful way, approximate $\square_\g$ by the flat
D'Alembertian $\square$.  To deduce a formula analogous to \eqref{eq:YM2}
 one might try to proceed by the geometrics optics
construction of parametrices for $\square_\g$, as developed in \cite{Fried}.  
Such an approach would require additional bounds on the background geometry,
determined by the metric $\g$, incompatible with the limited assumption 
{\bf A2}  and   the implied
finiteness of the curvature flux. 
 We   rely  instead  on  a  geometric version, which we develop in
\cite{Kirch}, of
 the  Kirchoff-Sobolev formula,  in the spirit of   that
used by Sobolev in \cite{Sob} and Y. Choquet-Bruhat in
\cite{Br}\footnote{It is extremely important that the error term generated by our parametrix depends only on the geometry of the boundary of the causal past of a point. This feature is  absent in all previous constructions.}.  
Applying that formula to equation  \eqref{eq:GR1} we 
obtain the following analogue 
of the formula \eqref{eq:YM2}:
\bea
\R(p)=-\int_{\NN^{-}(p;\de)}\A\,\c(\R\star\R)  + \EE
+\int_{\NN^{-}(p;\de)} \err\c\R\label{eq:GR2}
\eea
where $\A$ is a $4$-covariant 4-contravariant tensor 
defined as a solution of  a transport equation  along $\NN^{-}(p,\de)$
with appropriate (blowing-up) initial data at the vertex $p$, $\NN^{-}(p;\de)$ denotes the portion 
of the null boundary $\NN^{-}(p)$ in the time interval
 $[t(p)-\de, t(p)]$  and the error term $\err$ depends only on  the
extrinsic  geometry of   $\NN^{-}(p;\de)$. 
The term $\EE$  depends, in principle,  only on the properties of the space-time in the interval $[t(p)-\de, t(p)-\de/2]$. 

{\bf 4.}\quad
As in the Yang-Mills setting the structure of the term $\R\star \R$ allows
us to estimate one of the curvature terms by the flux of curvature:
\bea
|\int_{\NN^{-}(p,\de)}\A\,\c(\R\star\R) |&\les& \FF^-(p)\, \c\, \|\R\|_{L^\infty(\NN^{-}(p,\de))} 
\c\|\A\|_{L^2(\NN^{-}(p,\de))}\label{eq:GR33}\\
&\les& \de^{1/2} \,\c \FF^-(p) \,\c \|\R\|_{L^\infty(\NN^{-}(p,\de))}, \nn
\eea
provided that, 
\be{eq:GR333}\|\A\|_{L^2(\NN^{-}(p;\de))}\les \de^{1/2}.
\end{equation}
Negelecting, for a moment,  the third integral in \eqref{eq:GR2} 
we can thus expect to 
prove a result analogous
 to that  of Lemma \ref{le:YM2},
see  proposition \ref{prop:funny}. 
\begin{theorem}
\label{le:GR3} 
There exists a  sufficiently  small $\de>0$  and a large constant $C$, 
depending only on $\De_0$ in assumptions {\bf A1} and {\bf A2} as well as $\RR_0$ in
\eqref{eq:const-R0} such that for all $t_0\le t<t_*$,
\be{eq:main-sup-estim}
\|\R(t)\|_{L^\infty}\les 
\de^{-1}C\sup_{t-2\de\le  t'\le t-\de/2}\big(\|\R(t')\|_{L^2}+\|\D\R(t)\|_{L^2}+ \|\D^2\R(t)\|_{L^2}\big)
\end{equation}
\end{theorem}

{\bf 5.}\quad
 The proof of  \eqref{eq:GR33} 
depends on verifying \eqref{eq:GR333}.
In addition, to estimate the third term in \eqref{eq:GR2}, we need
to provide   estimates for tangential  derivatives of $\A$ and other geometric quantities
asssociated to the null hypersurfaces $\NN^{-}(p)$. In particular, it 
requires showing that $\NN^{-}(p) $ remains a {\it smooth}
(not merely Lipschitz) hypersurface in the time slab $(t(p)-\de, t(p)]$ for some  $\de>0$
   dependent only on  the constants $\De_0$ and  $\RR_0$. 
Thus to prove the desired theorem we have to show that all geometric 
quantities, arising in the parametrix construction,
can be estimated only in terms of the flux of the curvature $\FF_p^{-}$
along $\NN^{-}(p)$ and our main  assumption  {\bf A1}. 
Yet, to start with, it is not even clear
that we can provide a lower bound for the radius of injectivity
of $\NN^{-}(p)$. In other words the congruence of  null geodesics, initiating
at $p$, may not be controllable\footnote{Different null geodesics of the congruence
may intersect, or  the congruence itself may have  conjugate points,  arbitrarily
close to $p$.}
 only in terms of the
curvature flux.  Typicaly, in fact, lower bounds for the  radius of conjugacy
of a null hypersurface in a Lorentzian manifold are only available
 in terms of the sup-norm
of the curvature tensor $\R$ along the hypersurface, while the problem of 
short, intersecting, null geodesics appears not to be fully understood even
in that context. The situation is similar to that in  Riemannian geometry, exemplified 
by the Cheeger's theorem, where 
pointwise bounds on sectional curvature are sufficient to control the radius 
of conjugacy but to prevent the occurrence of short geodesic loops one needs 
to assume in addition an upper bound on the diameter and a lower bound on
the volume of the manifold.

In a sequence of papers, \cite{Causal1}--\cite{Causal3},
 see also \cite{Wang}\footnote{In 
\cite{Causal1}-\cite{Causal3}  we have considered the case 
of the congruence of outgoing future null geodesics initiating
on a $2$-surface  $S_0$ embedded in a 
space-like hypersurface $\Si_0$. 
The extension of our results to null cones from a point forms
the subject of Qian Wang's Princeton 2006 PhD thesis, see \cite{Wang}.}
   we have  proved 
lower bounds on the   geodesic  radius of conjugacy   of null   hypersurfaces. 
 The methods
  developed in those papers  can be  adapted   to  also    prove 
  lower bounds on  the  radius of  conjugacy   with respect to the 
  time  parameter\footnote{The results in  \cite{Causal1}-\cite{Causal3} and \cite{Wang} 
  were proved with respect to the geodesic foliation. In this  paper, as well
  as in \cite{injectivity},
   we rely on an extension     of  these  results 
   to the foliation on $\NN^{-}(p) $ induced by the space-like foliation $\Si_t$.}  $t$.   It may however be possible that  the radius of conjugacy
of the null congruence is bounded from below and yet there are
past null geodesics  form a point $p$ intersecting again at points arbitrarily close  with respect to the time  parameter $t$),    to $p$. In  \cite{injectivity}
we have shown that this cannot happen in a space -time verifying our 
conditions {\bf A1} and {\bf A2}. Thus the combined  results of \cite{Causal1}-\cite{injectivity}  allow us to derive a lower bound on the radius of injectivity  
of  $\NN^{-}(p)$ depending only on  $\De_0$.

{\bf 6.}\quad As in the case of  Yang-Mills equations
one can use the result of Lemma \eqref{le:GR3},
together with the classical local existence result
for the Einstein equations, such as that in \cite{C-K},  to show that solutions can be extended as long
the bounds  on $\piT$  hold true.

Finally we would like to point out possible refinements
of our main theorem \ref{thm:main}. We expect that one should be able to replace 
the pointwise  condition {\bf A2}  with the integral condition,
\be{eq:criterion-conj1}
\int_{t_0}^{t_*}  \|\pi(t)\|_{L^\infty}^2 dt <\infty
\end{equation}
Moreover it may be possible to improve the result even further by
eliminating the term $\nab \log n$ in \eqref{eq:criterion} or \eqref{eq:criterion-conj1}
and requiring instead only a  pointwise bound on $n$.

\section{Constant Mean Curvature foliations}
As described in the introduction  $({\mathcal M}_*, \g)$ denotes
a   Lorentzian  manifold 
of the form 
${\mathcal M}_*=I \times \Si$, where $\Si$ is
a three dimensional, compact, connected,
orientable smooth manifold foliated  by a CMC foliation
$\Si_t$   with lapse  $n$ and second fundamental
 $k$,
$$n=\big(-\g(\D t, \D
t)\big)^{-1/2},\qquad k(X,Y)=\g(\D_X\T, Y)$$
where   $\T$  denotes  the future 
unit normal to $\Si_t$. The time interval $I=[t_0,t_*)$, where $t_0=-1$ and
$t_*<0$.

We decompose a  space-time  vectorfield  $X$ 
 relative to  the unit timelike $\T$,
\be{eq:decomposeX}X=X^0 \T + \underline X,\qquad <\T,\Xb>=0,
\end{equation}
We   define the  positive definite Riemannian metric,
\be{eq:riem-metric}
h(X,Y)=X^0\c Y^0+g(\Xb, \Yb).
\end{equation}
where $g$ denotes the metric induced on $\Si_t$.
We  can also write   \eqref{eq:riem-metric} in the form,
\bea
h_{\a\b}=\g_{\a\b}+2\T_\a\T_\b.
\eea
 Given a space-time tensor $U$ we denote
by $|U|$ its  norm  with respect to the metric $h$.
More precisely, if $U$ is a $m-$ covariant tensor,
\bea
|U|^2=h^{i_1j_1} \cdots h^{i_m j_m}U_{i_1\ldots i_m}U_{j_1\ldots j_m}\label{norm-U}
\eea
The following bound follows immediately from  our main assumption 
\eqref{eq:criterion-finite}, 
\bea
  |\D\T|\  \les  \De_0\label{bound.DT}
\eea
Since  $
\D_\ga h_{\a\b}=2(D_\ga \T_\a\T_\b+\T_\a\D_\ga\T_\b)
$  we have 
 $|\D h|\le 4|\D \T|$. Therefore,
\bea
|\D h|\les \De_0 \label{bound.Dh}
\eea
Also, since the  components  of the deformation tensor 
$\pi=\piT=\Lie_\T \g$ are given by,
 \beaa
\pi_{00}=0,\quad\pi_{0i}=n^{-1} \nab_i n,\quad \pi_{ij}=n^{-1}\pr_t
g_{ij}=-2k_{ij},
\eeaa
we   have, 
\bea
|\piT|\les   \De_0\label{bound.pi}
\eea

Given two tensors $U,V$   we  shall   denote  by $U\c V$
any tensor which is obtained from the tensor product of $U$ and $V$
by taking contractions with respect
to the space-time  metric $\g$. Clearly,
$$|U\c V|\le |U|\c |V|.$$

For any coordinate chart ${\cal O}$, with coordinates $x=(x^1, x^2, x^3)$,
we denote by $(x^0=t, x^1, x^2, x^3)$ the transported coordinates
on  $I\times {\cal O}$ obtained by following the integral
curves of $\T$.
In these coordinates the metric $\g$ 
 takes the 
form 
\be{eq:g-transported}
\g=-n^2 dt^2 + g_{ij} dx^i dx^j,
\end{equation}
  Relative
to these coordinates
$t,x$ we have the equations,
 \bea
\pr_t g_{ij}&=&-2nk_{ij}\label{eq:prt-g}\\
\pr_t k_{ij}&=&-\nab_i\nab_j n
+n(R_{ij}+\tr_g k k_{ij}- 2k_{ia}
k^{a}\,_{j})\label{prt-k}
\eea
with $R_{ij}$ the Ricci curvature of 
of the induced metric $g$ on
 $\Si_t$.  We also have the constraint
equations,
\bea
R-|k|^2+(\tr k)^2&=&0\label{eq:constr1}\\
\nab^j k_{ij}=\nab_i \tr k
\label{eq:constr2}.
\eea

In view of the constant mean curvature condition
on the foliation $\Si_t$ we can always reparametrize $t$ so that
\be{eq:CMC}
\tr_g k=t.
\end{equation}
As mentioned in the introduction we can assume that the initial hypersurface $\Si_0$ 
corresponds to the value $t=t_0=-1$.
In view of  \eqref{prt-k}, 
\eqref{eq:constr1} and \eqref{eq:CMC} we deduce the lapse
equation,
\be{eq:lapse}
\lap n=|k|^2 n-1
\end{equation}
At a point $p$  of minimum for $n$
we must have $|k|^2 n-1\ge 0$.
Therefore, at $p$
$n\ge |k(p)|^{-2}$. On the other hand,
since $|k|^2 =|\kh|^2 +\frac{1}{3}(\tr
k)^2$, at a point of maximum   we
have,
$|\kh|^2n+\frac{1}{3}( \tr k)^2n-1\le 0$.
Therefore,
\begin{equation}\label{eq:lapse-ineq}
\frac{1}{\|k(t)\|_{L^\infty}^2}\le n\le
\frac{3}{t^2}.
\end{equation}
Observe also that, since $ \pr_t \log  (\det g)=-2n \tr k=-2nt$,
\beaa
\frac{d}{dt}|\Si_t|=\frac{d}{dt}\int_{\Si_0}\sqrt{\det g} dx
= -\int_{\Si_{0}}nt \sqrt{\det g}dx.
\eeaa
where $|\Si_t|$ denotes the volume of the compact manifold $\Si_t$.
Thus,
\beaa
0\le \frac{d}{dt}|\Si_t|\le 3|\Si_t||t|^{-1}
\eeaa
As a consequence of \eqref{eq:CMC} the ratio of the
 volumes of  $|\Si_t|$ and $|\Si_0|$ can be estimated 
by,
\beaa
1\le \frac{|\Si_t|}{|\Si_{0}|}\le   \frac{|t_0|^3}{|t|^3}
\eeaa

Therefore, since $t_0=-1$, we have proved,
\begin{proposition} \label{prop:max-princ}
 For all $-1=t_0\le t< t_*<0$ we have the bounds,
\be{eq:max-n}
\frac{1}{\|k(t)\|_{L^\infty}^2}\le n\le
\frac{3}{t^2}
\end{equation}
Moreover, if $|\Si_t|$ denotes the volume of $\Si_t$ and $\Si_0=\Si_{t_0}$,
\be{eq:vol-t}
|\Si_0|\le|\Si_t|\le   \frac{1}{|t|^3}|\Si_0|\les |t|^{-3}
\end{equation}
\end{proposition}
\subsection{Coordinate estimates}
 
We  recall  the following lemma, see lemma 2.2   in  \cite{injectivity}.
\begin{lemma} \label{le:initial-rho} If $\Si_0$ is compact and verifies {\bf A1}
of the introduction,  there
must exist a number $\rho_0>0$ such that every point $y\in \Si_0$  admits a neighborhood
$B$, included in a neighborhood chart   $U$, such that $B$ is precisely the Euclidean 
ball $B=B_{\rho_0}^{(e)}(y)$ relative to the local  coordinates  in $U$.
\end{lemma}
\begin{proof}:\quad 
See the proof of Lemma 2.2. in \cite{injectivity}. 
\end{proof}
Next   we  recall    the result of proposition  4.1. in  \cite{injectivity}.
\begin{proposition}
\label{prop:compare-metrics}
 If  assumptions  {\bf A1} and {\bf A2} are verified,
 then
there exists a large constant $C=C(\De_0)$ such that,
\be{eq:assum-2-at-t}
 C^{-1}|\xi|^2\le g_{ij}(t,x) \xi^i\xi^j \le  C|\xi|^2, \qquad \forall
x\in U
\end{equation}
\end{proposition}
\begin{proof}:\quad For convenience  we  reproduce  the proof  given
in \cite{injectivity}.
 We fix a coordinate chart $ U$  and consider
 the transported coordinates $t,x^1,x^2,x^3$ on $I\times U$. Thus 
$
\pr_t g_{ij} = -\frac 12 n\, k_{ij}.
$
Let $X=X$ be a time-independent vector on $\M$ tangent
to $\Si_t$. Then,
$$
\pr_t g(X,X) = -\frac 12 n\, k(X,X).
$$
Clearly,
\beaa
|n k(X,X)|\le|nk|_g|X|_g^2\le \|nk(t)\|_{L^\infty} |X|_g^2
\eeaa
with $|k|_g^2=g^{ac} g^{bd}k_{ab} k_{cd}$ and $|X|_g^2=X^i X^j g_{ij}=g(X,X)$.
Therefore, since $\pr_t |X|_g^2=\pr_t g(X,X)$,
$$
-\frac 12  \|n k(t)\|_{L^\infty} |X|_g^2 \le \pr_t |X|_g^2 \le \frac 12   \|n
k(t)\|_{L^\infty}  |X|_g^2. 
$$
Thus, 
\beaa
 |X|_{g_0} e^{-\int_{t_0}^t\|nk(\tau)\|_{L^\infty}  d\tau} \le |X|^2_{g_t} \le  |X|_{g_0}
e^{\int_{t_0}^t\|nk(\tau)\|_{L^\infty}d\tau}
\eeaa
from which \eqref{eq:assum-2-at-t} immediatley follows. 
\end{proof}
\subsection{ Sobolev inequalities}
The properties of local  transported coordinates
 established in the previous section  can be used to prove the following
Sobolev  inequality for scalar functions.
\begin{proposition} Assume assumptions {\bf A1} and {\bf A2} verified. There
exists a constant $C$ depending only on  $\De_0$ such for every smooth
scalar function on $\Si_t$, $t_0\le t<t^*$  such 
that
\be{eq:Sob-scalar1}
\|f\|_{L^3(\Si_t)}\le  C\big(\|\nab f\|_{L^1(\Si_t)}+\|f\|_{L^1(\Si_t)}\big)
\end{equation}

\end{proposition}
\begin{proof}:\quad By a partition of unity we may assume that  
 $f$   has compact support in  a local chart   $V=\Si_t\cap (I\times U)$ of
 transported coordinates  
 $ t, x=(x^1,x^2, x^3) $.
 Then, writing 
$$f(x)=\int_{-\infty}^{x^1}\pr_1 f(y,x^2, x^3) dy=\int_{-\infty}^{x^2}\pr_2 f(x^1,y, x^3)
dy =\int_{-\infty}^{x^3}\pr_1 f(x^1,x^2,y)
dy,$$
\beaa
|f(x)|^{3/2}
&\le&\bigg(\int_{-\infty}^{x^1}|\pr_1 f(y,x^2, x^3)| dy\c
 \int_{-\infty}^{x^2}|\pr_2 f(x^1,y, x^3)| dy \int_{-\infty}^{x^3}|\pr_1
f(x^1,x^2,y)\big| dy\bigg)^{1/2}
\eeaa
Thus, by H\"older,
\beaa
\int_V |f(x)|^{3/2}dx\le(\int_V |\nab f(x)| dx)^{3/2}
\eeaa
Therefore,  since in view of \eqref{eq:assum-2-at-t} we have $C^{-3/2}\le \sqrt{|g|} \le
C^{3/2}$,
$$\big(\int_V |f(x)|^{3/2}\sqrt{|g|}  dx\big)^\f12\les
 \int_V |\nab f(x)|\sqrt{|g|} dx.$$
 which proves \eqref{eq:Sob-scalar1} as desired. Inequality \eqref{eq:Sob-scalar1}
is an immediate consequence of \eqref{eq:Sob-scalar1}.
\end{proof}
\begin{corollary}\label{corr:Sob} For any smooth  tensorfield $F$ on $\Si_t$ and
 any $2\le p\le 6$,
\be{eq:Sob-scalar2}
\|F\|_{L^p(\Si_t)}\le C\big( \|\nab
F\|_{L^2(\Si_t)}^{3/2-3/p}\|F\|_{L^2(\Si_t)}^{3/p-1/2}+\|F\|_{L^2(\Si_t)}\big)
\end{equation}
\end{corollary}
\begin{proof}:\quad We have,
\beaa
\|F\|_{L^p}^{2p/3}&=&\||F|^{2p/3}\|_{L^{3/2}}\le C\big(\|\nab
|F|^{2p/3}\|_{L^1}+\||F|^{2p/3}\|_{L^1}\big)\\
&\le&C\big(\|\nab F\|_{L^2}+ \|F\|_{L^2}\big)   \c\||F|^{2p/3-1}\|_{L^2}\\
&\le&C\big(\|\nab F\|_{L^2}+ \|F\|_{L^2}\big)\c
(\|F\|_{L^{\frac{4p-6}{3}}})^{\frac{4p-6}{6}}
\eeaa
In the particular case when $p=6$ we derive,
\beaa
\|F\|_{L^6}^4\le C\big(\|\nab F\|_{L^2}+ \|F\|_{L^2}\big)\c\|F\|_{L^6}^{3}
\eeaa
Therefore,
\beaa
\|F\|_{L^6}\le C\big(\|\nab F\|_{L^2}+ \|F\|_{L^2}\big)
\eeaa
Similarly, for $p=3$,
\beaa
\|F\|_{L^3}^{2}\le C\big(\|\nab F\|_{L^2}+ \|F\|_{L^2}\big)\c\|F\|_{L^2}
\eeaa
and thus,
\beaa
\|F\|_{L^3}\le C\big(\|\nab F\|_{L^2}+ \|F\|_{L^2}\big)^{1/2}\c\|F\|_{L^2}^{1/2}
\eeaa
The general case  follows  by interpolation.
\end{proof}
Here is another useful simple  calculus inequality which
we will make use of.
\begin{lemma}
\label{le:simple-calc}
Let $F$ be a tensorfield on a compact Riemanian manifold.
Then,
\beaa
\|\nab F\|_{L^4}\le \|\nab^2 F\|_{L^2}^{1/2} \|F\|_{L^\infty}^{1/2}
\eeaa
\end{lemma}
\begin{proof}:\quad  After an  integration by parts and H\"older,
\beaa
\int_{\Si}|\nab F|^4\le \|\nab^2 F\|_{L^2}\|\nab F\|_{L^4}^2\|F\|_{L^\infty}
\eeaa
 Hence,
\beaa
\|\nab F\|_{L^4}^2\le \|\nab^2F\|_{L^2}\|F\|_{L^\infty}
\eeaa
\end{proof}
{\bf Remark.}\quad We cannot use transported coordinates to derive
a Sobolev inequality of the form,
\beaa
\|f\|_{L^\infty}\les \|\nab^2 f\|_{L^2}+\|f\|_{L^2}
\eeaa
even in the case of  a scalar function $f$. Indeed,
the standard Sobolev inequality in a  coordinate chart $U$ provides,
\beaa
\|f\|_{L^\infty(U)}\les\sum_{i,j=1}^3\|\pr_i\pr_j f\|_{L^2(U)}+\|f\|_{L^2(U)}
\eeaa
On the other hand $\nab_i\nab_j f= \pr_i\pr_j f-\Ga^l_{ij}\pr_l f$ and therefore
we cannot derive the desired  estimate without a bound for the $L^3$ norm 
of  $\Ga$. Unfortunately, the only way to estimate $\Ga$ is by
differentiating the equation $\pr_t g=-2n k$ from which we could only
bound its $L^2$ norm. To get around this difficulty we need a better system
of coordinates. In \cite{injectivity} we have a proved a slightly more general
version of the following:
\begin{theorem}\label{thm:harmonic}
Assume that $\MM_*$ is globally hyperbolic
and  verifies the assumptions {\bf A1} and {\bf A2} as well
as \eqref{eq:const-R0}. 
Then,   for any $\ep>0$, there exists $r_0>0$, depending only  on
 $\ep, \De_0, \RR_0, t_*$,
such that  on any geodesic ball   $B_{r}\subset \Si_{t}$, $r\le r_0$,
centered at a point $p_t\in \Si_{t}$,  there exist local coordinates relative to which
 the metric $g_t$ verify conditions
\bea
(1+\ep)^{-1}\de_{ij}\le g_{ij}&\le &(1+\ep) \de_{ij}\label{eq:harm-coord1}\\
r\int_{B_r(p)} |\pr^2 g_{ij}|^2 dv_g &\le&\ep.  \label{eq:harm-coord2}
\eea 
\end{theorem}
As a corollary we derive the following version
of the Sobolev  inequlity
\begin{corollary}\label{cor:Sobolev}
Given a  smaooth scalar  function f  on $\Si_t$ we have,
\bea
\|f\|_{L^\infty(\Si_t)}\le C\big\|\nab^2 f\|_{L^2(\Si_t)}+\|f\|_{L^2(\Si_t)}\big)
\eea
with $C>$ a universal constant, i.e.  depending only on  the fundamental constants
$ \De_0,\RR_0, t_* $.
\end{corollary}

\section{Basic  Curvature energy estimates}
\subsection{General procedure}\label{sec:energy estimates}We   recall the general procedure
to derive  energy estimates for $\R$, see section 7.1 in \cite{C-K}. First let $W$ denote a
Weyl field, i.e  a four covariant tensor traceless tensor  $W_{\a\b\ga\de}$ verifying all
the algebraic symetries of  the   curvature tensor $\R$. Let, 
\be{eq:Bel-Robinson-W}
\Q[W]_{\a\b\ga\de}=W_{\a\la\ga\mu}W_{\b\,\,\de}^{\,\la\,\,\mu}+
\dual W_{\a\la\ga\mu}\dual W_{\b\,\,\de}^{\,\la\,\,\mu}\,\,
\end{equation}
Given  a vectorfield $X$ we denote 
$P_\a=\Q[W]_{\a\b\ga\de}X^\b X^\ga X^\de$.
By a straightforward calculation,
\beaa
\D^\a P_\a=\Div \Q[W]X^\b X^\ga X^\de+
\frac{3}{2}\Q_{\a\b\ga\de}\piX^{\a\b}X^\ga X^\de
\eeaa
where $\piX$ is the deformation tensor of $X$
Therefore,
integrating on the slab $\cup_{t'\in[t_0,t]}\Si_{t'}$ 
we derive the following. 
\begin{proposition}
\label{prop:main-energy}Let $\Q=\Q[W]$ be the Bel-Robinson
tensor of a  Weyl field $W$. Then,
\bea
\int_{\Si_t}\Q(X,X,X, \T)
&=&\int_{\Si_{0}}\Q( X,X,X,\T)
+\int_{t_0}^t\int_{\Si_{t'}}\Div \Q(X,X,X) n dv_g\nn\\
&+&\frac{3}{2}\int_{t_0}^t\int_{\Si_{t'}}\Q_{\a\b\ga\de}\piX^{\a\b}X^\ga
X^\de ndv_g
\eea
with $dv_g$ denoting the volume element on $\Si_t$.
\end{proposition}
The following proposition is an immediate consequence,  see also section 5
 in  \cite{injectivity} for a proof. One simply needs to apply the proposition above
 for $X=\T$ together with  the positivity of $\Q(\T,\T,\T,\T)$ and the  uniform
 bounds for $\piT$ and $n$. 
\begin{proposition}
\label{prop:L2-curv}
 Under assumption {\bf A2}     There exists a constant $C=C( \De_0, t_*)$ 
such that,  for any $t_0\le t< t_* <0$,
\be{eq:L2-curv}
\|\R(t)\|_{L^2}\le C \RR_0.
\end{equation}
where $\RR_0$ is the constant defined by \eqref{eq:const-R0}.
\end{proposition}
\begin{definition}\label{def:notation} In what follows
 we extend the  usual notation   
$A\les B$ to include inequalities  $A\le  c B$ where $c=c(t_*,\De_0, \RR_0)$ is a
constant which  depends    on  our fundamental constants $t_*$,  $\De_0$ and $\RR_0$.
\end{definition}
 In
particular, in view of proposition \ref{prop:L2-curv}  we can  write 
$$\|\R(t)\|_{L^2}\les \RR_0\les 1.$$

\subsection{Wave equation for the curvature tensor}
Recall the Bianchi identitities,
\be{eq:Bianchi}
\D_{[\si}\R_{\a\b]\ga\de}=0
\end{equation}
or, equivalently since $\R_{\a\b}=0$,
\be{eq:Bianchi2}
\D^{\de}\R_{\a\b\ga\de}=0
\end{equation}
Differentiating  \eqref{eq:Bianchi} once more an taking the trace we derive,
\beaa
\Box \R_{\a\b\ga\de} +\D^\si \D_\a \R_{\b\si\ga \de}+ \D^\si \D_\b \R_{\si\a\ga \de}=0
\eeaa
Now, in view of \eqref{eq:Bianchi2}, commuting
covariant derivatives, 
\beaa
\D^\si \D_\a \R_{\b\si\ga \de}&=& \R_{\b\,\,\,\,\,\a }^{\,\,\, \mu\si}\R_{\mu\si\ga\de}
+\R_{\ga\,\,\,\,\,\a }^{\,\,\, \mu\si}\R_{\b\si\mu\de}+\R_{\de\,\,\,\,\,\a }^{\,\,\,
\mu\si}\R_{\b\si\ga\mu}\\
\D^\si \D_\b \R_{\a\si\ga \de}&=& \R_{\a\,\,\,\,\,\b }^{\,\,\, \mu\si}\R_{\mu\si\ga\de}
+\R_{\ga\,\,\,\,\,\b }^{\,\,\, \mu\si}\R_{\a\si\mu\de}+\R_{\de\,\,\,\,\,\b }^{\,\,\,
\mu\si}\R_{\a\si\ga\mu}
\eeaa
Hence,
\beaa
\D^\si \D_\a \R_{\b\si\ga \de}+\D^\si \D_\b \R_{\si\a\ga \de}&=&
\R_{\mu\si\ga\de}\big(\R_{\b\,\,\,\,\,\a }^{\,\,\, \mu\si}-
\R_{\a\,\,\,\,\,\b }^{\,\,\, \mu\si}\big)\\
&-&\R_{\a\si\ga\mu}\big(\R_{\de\,\,\,\,\,\b }^{\,\,\,
\mu\si}+\R_{\b\,\,\,\,\,\de }^{\,\,\,
\si\mu}\big)\\
&-& \R_{\a\si\mu\de}\big(\R_{\b\,\,\,\,\,\ga }^{\,\,\, \si\mu}+\R_{\ga\,\,\,\,\,\b }^{\,\,\,
\mu\si}\big)
\eeaa
Thus introducing the notation,
\bea
(\R\star \R)_{\a\b\ga\de}&=&-\R_{\mu\si\ga\de}\big(\R_{\b\,\,\,\,\,\a }^{\,\,\, \mu\si}-
\R_{\a\,\,\,\,\,\b }^{\,\,\, \mu\si}\big)
+\R_{\a\si\ga\mu}\big(\R_{\de\,\,\,\,\,\b }^{\,\,\,
\mu\si}+\R_{\b\,\,\,\,\,\de }^{\,\,\,
\si\mu}\big)
\label{eq:notation1}\\
&+&\R_{\a\si\mu\de}\big(\R_{\b\,\,\,\,\,\ga }^{\,\,\, \si\mu}+\R_{\ga\,\,\,\,\,\b
}^{\,\,\,
\mu\si}\big)\nn
\eea
we derive,
\be{eq:wave1}
\square \,\R=\R\star \R
\end{equation}
Clearly $W=\R\star \R$ is a Weyl field, i.e.
it satisfies all the algebraic symmetries of the curvature
tensor  plus  the traceless condition  $W_{\a\,\,\mu \b}^{\,\,\mu}=0.$
\subsection{Energy estimates for higher derivatives}
To  estimate   the first derivatives of $\R$ we shall use the covariant  wave equation  \eqref{eq:wave1}. 
Recall  the positive definite  space-time metric $h$ defined by \eqref{eq:riem-metric}.  Given a tensor-field  $U_{\a_1\ldots \a_m}$
we write, for simplicity,
\beaa
h^{IJ} U_I U_J&=& h^{\a_1\b_1}\ldots h^{\a_m\b_m}U_{\a_1\ldots\a_m }U_{\b_1\ldots\b_m }\\
U_I&=&U_{\a_1\ldots\a_m }, \quad U_J=U_{\b_1\ldots\b_m },\quad h^{IJ} =h^{\a_1\b_1}\ldots h^{\a_m\b_m}
\eeaa
Consider the energy-momentum type   tensor $\Q^{(w)}_{\a\b}$ associated with the covariant wave operator 
$\square$ acting on tensors,
\bea
\Q^{(w)}[U]_{\a\b}:&=&h^{IJ} \D_\a U_I \D_\b U_J-\frac 1 2 \g_{\a\b} h^{IJ} \g^{\mu\nu}\D_\mu U_I\D_\nu U_J
\eea
We have,
\beaa
\D^\b\Q^{(w)}[U]_{\a\b}&=&h^{IJ} \D_\a U_I  \D^\b \D_\b U_J+h^{IJ}  \D^\b\D_\a U_I \D_\b U_J-
 \g_{\a\b}h^{IJ} \g^{\mu\nu} \D^\b\D_\mu U_I\D_\nu U_J\\
 &+& D^\b  h^{IJ} \D_\a U_I \D_\b U_J-\frac 1 2 \g_{\a\b} D^\b h^{IJ} \g^{\mu\nu}\D_\mu U_I\D_\nu U_J\\
 &=&h^{IJ} \D_\a U_I ( \square U_J) +h^{IJ}  (\D_\b \D_\a U_I-  \D_\a \D_\b U_I)    \D^\b U_J\\
 &+& D^\b  h^{IJ}\big ( \D_\a U_I \D_\b U_J-\frac 1 2 \g_{\a\b} \g^{\mu\nu}\D_\mu U_I\D_\nu U_J\big)
\eeaa
Consequently, in view of \eqref{bound.Dh},
\beaa
|\D\Q^{(w)}[U]|\les |\D U||\Box U|+|\R| |U| |\D U|+ \De_0 |\D U|^2
\eeaa
Therefore
since,
\bea
\D^\b(\Q^{(w)}[U]_{\a\b}\T^\a)=D^\b\T^\a\Q^{(w)}[U]_{\a\b}+T^\b \D^\b\Q^{(w)}[U]_{\a\b}
\label{eq:divQ.wave}
\eea
we derive,
\bea
|\D^\b(\Q^{(w)}[U]_{\a\b}\T^\a)|&\les&  |\D U||\Box U|+|\R| |U| |\D U|+\De_0 |\D U|^2
\label{eq:divQ.wave-abs}
\eea
On the other hand,
\beaa
\Q^{(w)}[U](\T,\T)&=&\frac 1 2  h^{IJ}\big(  \D_0 U_I \D_0 U_J+\sum_{l=1}^k\D_l U_I\D_l U_J \big)\\
&=&\frac 12 |\D U|^2
\eeaa
Integrating \eqref{eq:divQ.wave} we derive, 
\bea
\int_{\Si_t} |\D U|^2 &=&  2 \int_{\Si_t}\Q^{(w)}[U](\T,\T)\nn \\
&\le &   \int_{\Si_{t_0}} |\D U|^2+
     \int_{t_0}^t\int_{\Si_{t'} }|\D^\b(\Q^{(w)}[U]_{\a\b}\T^\a)|\nn\\
     &\les&  \int_{t_0}^t\int_{\Si_{t'}}   \left (|\D U||\Box U|+|\R| |U| |\D U|+\De_0 |\D U|^2\right).\label{eq:divQ.wave-int}
     \eea
Applying this to $U=\R$ and using the equation \eqref{eq:wave1} we obtain,
\beaa
\|\D\R(t)\|^2_{L^2}&\les& \|\D\R(t_0)\|^2_{L^2}+ \De_0  \int_{t_0}^t\|\D\R(t')\|^2_{L^2} dt'\\&+& 
\int_{t_0}^t\|\D\R(t')\|_{L^2}\|\R(t')\|_{L^2}\|\R(t')\|_{L^\infty} dt'\\
&\les&  \|\D\R(t_0)\|^2_{L^2}+\De_0  \int_{t_0}^t\|\D\R(t')\|^2_{L^2} dt'\\ &+&\RR_0   \int_{t_0}^t\|\D\R(t')\|_{L^2} \|\R(t')\|_{L^\infty} dt'
\eeaa
Therefore in order to get an a-priori estimate for  $ \|\D\R(t)\|_{L^2}$ it  suffices to prove an estimate
for the $L^\infty$ norm of $\R$.  More precisely,
\begin{proposition}\label{thm:deriv-energy-estim}
Assume that  the assumptions
{\bf A1}, {\bf A2} hold true. 
Then the following  derivative curvature estimates hold true
for all $t_0\le t< t_*$,
\bea
\|\D\R(t)\|_{L^2}^2\le C \big(\|\D\R(t_0)\|_{L^2}^2+\int_{t_0}^t \|\R(t')\|_{L^\infty}^2  dt' \big)
\eea
with $C$ a constant depending only on $\De_0$, $\RR_0$ and  $t_*$. 
\end{proposition}
To estimate the second   derivatives  of $\R$  we apply  \eqref{eq:divQ.wave-int}
to the tensor $U=\D\R$. Thus,
\beaa
\int_{\Si_t} |\D^2 \R |^2 
     &\les&  \int_{t_0}^t\int_{\Si_{t'}}   \left (|\D^2\R|\,  |\square( \D\R)|+|\R| |\D\R| |\D^2\R|+\De_0 |\D^2\R|^2\right).\\
     &\les& \int_{t_0}^t\int_{\Si_{t'}}   \left (|\R| |\D\R| |\D^2\R|+\De_0 |\D^2\R|^2\right)
     \eeaa
Hence,
\beaa
\|\D^2\R(t)\|^2_{L^2}&\les& \|\D^2\R(t_0)\|^2_{L^2}+ \De_0  \int_{t_0}^t\|\D^2\R(t')\|^2_{L^2} dt'\\&+& 
\int_{t_0}^t\|\D^2\R(t')\|_{L^2}\|\D\R(t')\|_{L^2}\|\R(t')\|_{L^\infty} dt'\\
&\les&  \|\D^2\R(t_0)\|^2_{L^2}+\De_0  \int_{t_0}^t\|\D^2\R(t')\|^2_{L^2} dt'\\ &+&   \int_{t_0}^t\|\D^2\R(t')\|_{L^2}\|\D\R(t')\|_{L^2} \|\R(t')\|_{L^\infty} dt'.
\eeaa
We therefore,
deduce the following,
\begin{proposition}
\label{thm:deriv-energy-estim2}
Under the same assumptions as in proposition 
\ref{thm:deriv-energy-estim2} we have,
\bea
\|\D^2\R(t)\|_{L^2}^2 &\le& C\big(  \|\D^2\R(t_0)\|_{L^2}^2+ \int_{t_0}^t\|\D\R(t')\|^2_{L^2}\|\R(t')\|_{L^\infty}^2 dt'\big)
\eea
with $C$ a constant depending only on $\De_0$, $\RR_0$ and  $t_*$. 
\end{proposition}

\section{Past null boundaries}\label{sec:past}
The goal of this section is to
review the main result of \cite{injectivity}, concerning  the null boundaries of past
domains of dependence,  and show how they apply to our situation.
Starting with any point $p$ in a subset  $\MM_*=\cup_{t\in [t_0, t_*)}\Sigma_t $ of $\M$,
 we denote by $\JJ^{-}(p)=\JJ^{-}(p; \MM_*)$ the causal past of $p$,
  relative to $\MM_*$,   by $\II^{-}(p)$ 
its  interior and by $\NN^{-}(p)$ its null boundary. 
 In general $\NN^{-}(p)$ is
an achronal, Lipschitz  hypersurface, ruled
by the set of past  null geodesics  from $p$. 
We parametrize these geodesics   with respect to the  future, unit, time-like vector
 $\T_p$. Then,  for every direction $\omega\in \SSS^2$,
 with  $\SSS^2$ denoting the standard sphere in $\RRR^3$,  consider the  
 null vector $\ell_\omega$ in $T_p\M$,  
\be{eq:norm-geodd}
\g(\ell_\omega, \T_p)=1,
\end{equation}
  and associate to it the  past null
geodesic $\ga_\omega(s)$ with initial data $\ga_\omega(0)=p$ and 
$\dot \ga_\omega(0)=\ell_\omega$. 
We further define a null vectorfield $L$ on $\NN^-(p)$ according to 
$$
L(\gamma_\omega(s))=\dot\ga_\omega(s).
$$
 $L$ may only be  smooth almost everywhere on $\NN^{-}(p)$ and 
can be multi-valued on a set of exceptional points. 
We can choose the parameter $s$ in 
such a way so that $L=\dot\ga_\om(s)$ is geodesic and $L(s)=1$.

 For a sufficiently small $\de>0$ the  exponential map  
 $\GG=\GG^{-}_p$  defined by, 
\be{eq:exp-map}
(s,\om)\to \ga_\om(s)
\end{equation}
is a  diffeomorphism from $(0,\de)\times \SSS^2$
to its image in $\NN^{-}(p)$.
 Moreover  for each $\om\in \SSS^2$ either $\ga_\om (s)$ can be
 continued for all positive 
  values of $s$\footnote{ for which $\ga_\om(s)$
   stays in $\MM_*$} or there exists a   value
$ s_*(\om)$ beyond which  the points $\ga_\om(s)$ are no longer on the boundary
 $\NN^{-}(p)$ of $\JJ^{-}(p)$ but rather in its interior, see \cite{HE}.  We call
such points terminal points of $\NN^{-}(p)$. We say that a terminal point $q=\ga_\om(s_*)$ 
is a conjugate terminal point  if the map  $\GG$ is singular at $(s_{*},\om)$. 
A terminal point  $q=\ga_\om(s_*)$ is said to be a cut locus terminal point if  
the map $\GG=\GG^-_p$ is nonsingular at  $(s_{*},\om)$ and there exists another 
null geodesic from $p$, passing through $q$.  

 Thus    $\NN^{-}(p)$ 
 is a smooth manifold at all points except the vertex $p$ and
 the  terminal points of its past null geodesic generators. We denote
by $\TT^{-}(p)$ the set of all terminal points and by 
$\Null(p)=\NN^{-}(p)\setminus \TT^-(p)$ the  smooth  portion of $\NN^-(p)$.
The set $\GG^{-1}(\TT^{-}(p))$
has measure zero relative to the standard  measure $ds da_{\SSS^2}$ of the cone
$[0, \infty)\times \SSS^2$. We will denote by $dA_{\NN^{-}(p)}$ the 
corresponding    measure on $\NN^{-}(p)$. Observe that the definition is not intrinsic,
it depends in fact   on the  normalization condition \eqref{eq:norm-geodd}. 

 \begin{definition}\label{def:rad-inj-t} 
  Given $p\in \MM_*$  we define $i_*^{-}(p)$  to be the supremum over
 all the values $s>0$ for which the exponential map 
$\GG_p^-:(s,\om)\to \ga_\om(s)$ is a global diffeomorphism.
We shall refer to $i_*^{-}(p)$ as the \emph{ past null  radius of injectivity} 
at  $p$ relative to the \emph{geodesic foliation }
defined by \eqref{eq:norm-geodd}. 

We  also  define $i_*^-(p, t)$  (the null  radius of injectivity relative to the $t$-foliation)    to be the
 supremum  over all the values $t(p)-t$, $t<t(p)$, for which the exponential map $\GG=\GG_{p,t}^-$, 
\bea
(t,\om)\to \ga_\om(t)=\ga_\om(s(t))
\eea
 is a global diffeomorphism. 
\end{definition}
 \begin{definition}
 We define $d^-(p,t)$ to be the distance, measured
 with respect to the time parameter $t$, from $p$ to the past boundary of $\MM_*\subset\M$.  \end{definition}

 The following theorem is an immediate 
consequence    of  the Main Theorem II 
proved  in \cite{injectivity}. 
\begin{theorem}\label{thm:injectivity}  Assume that $\MM_*$ is globally 
hyperbolic and verifies the assumptions
 {\bf A1} and {\bf A2} as well as \eqref{eq:const-R0}.  
There exists 
 a positive number
$i_{*}>0$, depending only on    $\De_0$,   $\RR_0$,  and $t_*<0$,
such that, for all $p\in \MM_*$, 
\be{eq:lower-bound-inj}
i_{*}^-(p, t)>\min(i_*, d^{-}(p,t))
\end{equation}
\end{theorem}
\begin{proof}:\quad 
According to  the Main theorem II of \cite{injectivity}
and the remark following it, 
 $i_*$ depends only on our main  constants,
 $\De_0$, $\RR_0$  and  a  constant 
$N_0$ which provides  uniform  bounds for the lapse  $n$,
\beaa
N_0^{-1}\le n\le N_0.
\eeaa
The finiteness of $N_0, N_0^{-1}$ follows from \eqref{eq:lapse-ineq} and the assumption  $t_*<0$.
 \end{proof}
 Once we have a lower bound  for $i_*^-(p,t)$ it is
 straightforward to also get a lower bound for  the radius of injectivity 
 $i_*^-(p)$  with respect to  the geodesic foliation. Indeed all we need is to show that $s$ does not vary  much (along $\NN^{-}(p)$)  as a function of $t$ in a time interval of size $1$. This follows immediately
 from the following.
 \begin{lemma}
  There exists a constant $c>0$,
depending only on  $\De_0$,  such that,
  \bea
  c^{-1}\le |\frac{dt}{ds}|\le c.\label{eq:dsdt}
  \eea
 \end{lemma} 
 \begin{proof}:
 We    introduce the null lapse,
 \be{eq:null-lapse}
\varphi^{-1}=g(\T,L)
\end{equation}
Observe that $\varphi>0$  with $\varphi(p)=1$. Moreover
\be{eq:dtds}
\frac{dt}{ds}=-(n\varphi)^{-1}
\end{equation}
with $n$ the lapse function
of the $t$ foliation. On the other hand,
we have 
\beaa 
L=-\vphi^{-1}(\T+N)
\eeaa
with $N$ of length $1$ perpendicular to $\T$. 
Now,
\beaa
\frac{d}{ds}\vphi^{-1} &=&\frac d{ds}\,\g(\T, L)=\g(\D_L \T, L)=-\f12\,
 \piT_{LL}\\
&=&\frac{1}{2}\vphi^{-2}( \,\piT_{\T N}+\f12\, \piT_{NN})).
\eeaa
Therefore,
\beaa
|\frac{d}{ds}\vphi ^{-1}| \les\vphi^{-2}  \De_0
\eeaa
from which,
\begin{equation}\label{eq:null-lapse-ineq}
|\vphi(s)-1|\les \De_0 s.
\end{equation}
Thus, for  an interval in $s$ of size $1$ we deduce that $2^{-1} \le \phi(s)\le 2$ 
and therefore, in view of the uniform  bound for $n$ of proposition
\ref{prop:max-princ},  we infer that  there must  exist a constant $c>0$,
depending only on  $\De_0$,  such that \eqref{eq:dsdt} holds.
  
 \end{proof}

\subsection{Geometry of smooth null cones}In this 
subsection we provide additional geometric informations
for  the null  boundaries  $\NN^{-}(p,\, \de)$ 
with $ \de<i^-_*(p,t)$ with $i_*^-(p,t)$  a lower bound for  past  null  injectivity radius 
 with respect to the $t$-foliation. Here  $\NN^{-}(p,\, \de)$  denotes the portion
of $\NN^{-}(p)$ for $t$  between $t(p)$ and $t(p)-\de$.

Let   $S_t$ denote  the $2$ dimensional
space-like  surfaces of intersection between $\Si_t$ 
and $\NN^-(p)$. 
At any point of  $\NN^{-}(p,\, \de)\setminus \{ p\}$ we can define a 
conjugate null vector $\Lb$ with $\g(L, \Lb)=-2$
and such that  $\Lb$ is orthogonal to the leafs   $S_t$. 
In addition we can choose  $(e_a)_{a=1,2}$ tangent  $S_t$
such that together with $L$ and $\Lb$ we obtain a null frame,
\begin{align}
&g(L,\Lb)=-2,\qquad \g(L,L)=\g(\Lb,\Lb)=0,\nn\\
&\g(L,e_a)=\g(\Lb,e_a)=0,\qquad \g(e_a,e_b)=\de_{ab}.\label{eq:null frame}
\end{align}
We denote by $\ga$ the  restriction 
of $\g$ to $S_{t}$  i.e.   $
 \ga(X, Y)=\g(X, Y) $
Endowed with this  metric  $S_t$ is a $2$ dimensional 
 compact riemannian manifold.
 We  denote by $\nabb$  the restriction of $\D$ to $S_t$,
Clearly, for all $X, Y\in T(S_t)$,
\beaa
\nabb_X Y=\D_X Y+\f12 \g(\D_X Y, \Lb) L+\f12 \g(\D_X Y, L) \Lb
\eeaa
We recall, see \cite{Causal1} the  definitions of the following basic geometric
quantities:
\begin{definition} The null second fundamental forms $\chi,\chib$, torsion $\ze$ and
the Ricci coefficient $\etab$ 
 of the foliation $S_t$ are defined as follows: 
\begin{align}
&\chi_{ab}=\g(D_aL\,,\, e_b),\qquad \chib_{ab}=\g(D_a\Lb\,,\, e_b), \label{eq:nullsforms}\\
&\ze_a=\f12 \g(D_a L\,,\,\Lb),\qquad \etab_a=\f12\g(e_a\,,\, D_L\Lb).\label{eq:torsion}
\end{align}
In addition we define $\trch=\ga^{ab} \chi_{ab}$, $\chih_{ab}=\chi_{ab}-\frac 12 \trch \ga_{ab}$
and
$$
\omega=-\frac 14\g(D_{\Lb} \Lb, L),\qquad
\mu=\Lb (\trch) +\frac 12 \trch \trchb +2\omega \trch
$$
\end{definition}
We note that 
\begin{align*}
&\chib_{ab}=-\varphi^2 \chi_{ab} +2\varphi k_{ab},\\
&\etab_a=-\ze_a- (n\varphi)^{-1} e_a\left (n\varphi\right),\\
&\omega=\varphi\, n^{-1} N(n).
\end{align*}
Our conventions imply that
\be
{eq:def-nab}\nabb_X Y=\D_X Y- \f12\chib(X,Y)L- \f12\chi(X, Y)\Lb
\end{equation}
We extend the definition of $\nabb$ to  any  covariant  $S-$ tangent tensor $\pi$  by the usual formula,
\beaa
\nabb_X\pi(Y_1,\ldots, Y_k)&=&X\pi(Y_1,\ldots, Y_k)-\pi(\nabb_XY_1,\ldots, Y_k)-\ldots -\pi(Y_1,\ldots, \nabb_XY_k)
\eeaa 
with $X, Y_1, \ldots Y_k$  $S$-tangent.
Given an $S$-tangent vector-field $X$  we define  $\nabb_LX$ to be
the projection to $S_t$ of $\D_L X$, 
\beaa
\nabb_L X=\D_L X+\f12 \g(\D_L X, \Lb)L
\eeaa
We extend  the definition  to  any  covariant  $S$-tangent tensor $\pi$  by 
\beaa
\nabb_L\pi(Y_1,\ldots, Y_k)&=&L\pi(Y_1,\ldots, Y_k)-\pi(\nabb_LY_1,\ldots, Y_k)-\ldots -\pi(Y_1,\ldots, \nabb_LY_k)
\eeaa 
with $ Y_1, \ldots Y_k$  $S$-tangent.

Given an $S$-tangent tensor $\pi$ wewrite $\na \pi=(\nabb \pi, \nabb_L\pi)$
and 
$$|\na\pi|^2=|\nabb_L\pi|^2+|\nabb\pi|^2.$$
\subsection{ Tangential covariant derivatives of space-time tensors}
In this section we make sense of covariant derivatives of space-time 
tensors, not necessarily $S$-tangent  along a fixed surface $S=S_{t}\subset \NN^{-}(p,\de)$. 

We start by defining a covariant derivative for space-time vector $A_\mu$
defined on $S$. Thus we view $A$ as a section of a vector bundle $T^*\M$ over $S$.
We interprete the covarant derivative $\nabb A$ of $A$ along $S$ as a 1-form 
on $S$ with values in $T^*\M$.  Thus, for every vectorfield $X\in TS$
and any vectorfield $Z$ in $ T\M$,
\beaa
\nabb A(X; Z)=\nabb_X A(Z)=X\big(A(Z)\big)-A(\D_X Z)=\D_XA(Z)
\eeaa
We also write,
$$
(\nabb_X A)_\mu = X^a \D_a A_\mu,\qquad \forall X\in TS. 
$$
 We  define $\nabb^2 A$,  the second   covariant derivatives of $A$ along $S$,
by the formula,
\beaa
\nabb^2 A(X, Y; Z)=(\nabb_X \nabb A)(Y;Z)=X(\nabb A(Y;Z))-\nabb A(\nabb_X Y; Z)-\nabb A(Y;\D_X Z)
\eeaa
or,  for simplicity,
$$
\nabb^2 A_\mu (X,Y) = (\nabb_Y (\nabb_X A))_\mu - (\nabb_{\nabb_Y X} A)_\mu
$$
These definitions can be easily extended to higher
covariant derivatives along $S$ and 
to higher order tensors $A$.

Given a  $A$  an $S$-tangent 1-form on $M$ with values in $TM$ we  define 
$$
\nabb_\L A (X;Y) = L (A(X;Y)) - A(\nabb_L X;Y) - A(X;\D_L Y),
\qquad \forall X\in TS,\quad Y\in TM.
$$
This defintion extends naturally to higher order tensors $A$. Note that 
for a scalar function $A$ on $M$  we have 
$$
\nabb_L A = \D_L A
$$
\subsection{Commutation formula}
In what follows we will need the following
commutation lemma, see \cite{Causal1}.

\begin{lemma}\label{lem:commute}
Let $A_{\mu}$ be a function on $M$ with values in $TM$ verifying the equation
\be{eq:trans-AF}
\nabb_\L A = F
\end{equation}
for some $TM$ valued function $F$. Then, 
\be{eq:trans-AF2}
\nabb_\L (\nabb_a A_\mu) +\chi_{ab} \nabb_b A_\mu= \nabb_a F_\mu + (\ze_a+\etab_a) F_\mu+
 {\R_{\mu}^{\,\,\la}}_{\L a} A_\la.
\end{equation}
\end{lemma}

\subsection{Curvature flux}
\begin{definition}The curvature  flux along\footnote{Given a scalar function $f$
on $\NN^-(p)$ we denote  its integral 
on   $\NN^-(p)$  to be,
 $\int_{\NN^-(p)} f=\int_{t_0}^{t(p)} n dt \int_{S_t} f dA_t=
\int_{\NN^-(p)} f \,\,  dA_{\NN^{-}(p)} .$ .} 
$\NN^-(p)$,  is defined   as follows.
\beaa
\FF(p)=\int_{\NN^-(p)}\Q[\R](\T,\T,\T, L)=\int_{t_0}^{t(p)} n  dt
\int_{S_t}\Q[\R](\T,\T,\T,L)dA_t
\eeaa
with  $dA_t$  the area element of $S_t$. We also let 
\beaa
\FF(p,\de)=\int_{\NN^-(p,\de)}\Q[\R](\T,\T,\T, L)=\int_{t(p)-\de}^{t(p)} n  dt
\int_{S_t}\Q[\R](\T,\T,\T,L)dA_t
\eeaa
to be the curvature flux along $\NN^-(p,\de)$ for $\de<i_*^-(p,t)$.
\end{definition}
The following is an immediate consequence of the energy estimates of section
\ref{sec:energy estimates}, see also \cite{injectivity}.
\begin{proposition}
\label{prop:flux}Under assumptions {\bf A2}  as well as \eqref{eq:const-R0} 
 the  flux of curvature  $ \NN^{-}(p) $  (denoted
  $\FF(p)$),
can be bounded by a uniform constant 
independent of $p$. More precisely, for all $p$ with $t_0<t(p)\le t_*<0$
\beaa
 \FF(p)\le C(t_*,\De_0) \RR_0.
\eeaa
where $C$ is the constant of proposition \ref{prop:L2-curv}
\end{proposition}
\begin{proof}
See section 5 in \cite{injectivity}.
\end{proof}

We  can allso introduce the reduced flux, or geodesic
curvature flux,   
\be{eq:rondR'}
\RR(p)=\big(\int_{t_0}^{t(p)} \int_{S_t}
|\a|^2+|\b|^2+|\rho|^2+|\si|^2+|\bb|^2\,\,\big)^{1/2}
\end{equation}
as well as 
\be{eq:rondR}
\RR(p,\de)=\big(\int_{t(p)-\de}^{t(p)} \int_{S_t}
|\a|^2+|\b|^2+|\rho|^2+|\si|^2+|\bb|^2\,\,\big)^{1/2}
\end{equation}
where $\a,\b, \rho, \si, \bb$ are the   null  components of the Riemann curvature tensor
 relative to  the $S_{t}$ foliation:
\bea
\a_{ab}&=&\R_{L a L b}\,,
\quad \b_a=\f12 \R_{a L\Lb L} ,\quad
\rho=\frac{1}{4}\R_{\Lb L \Lb L}\,\quad\nn\\
\quad
\si&=&\frac{1}{4}\, ^{\star} \R_{\Lb L\Lb L},\quad
\bb_a=\f12R_{a\Lb\Lb L},\quad \aa_{ab} =\R_{\Lb a\Lb b}\
\eea
\begin{proposition}\label{thm:geom-estim} Under the same assumptions
as in proposition \ref{prop:flux} we have,  with a constant $C$ depending
only on $t_*, \De_0$,
\bea\label{eq:Rp}
\RR(p)&\le&  C(t_*,\De_0 )\, \RR_0
\eea
\end{proposition}
\begin{proof}
We can express $L, \Lb$ in the form,
\beaa
L=\varphi^{-1}(\T+N),\qquad \Lb=\vphi(\T-N)
\eeaa
where $N$ is  unit normal of $S_t$ on $\Si_t$ and $\vphi$ the null lapse defined by \ref{eq:null-lapse}. Also,
\bea
\T=\vphi L+\vphi^{-1}\Lb
\eea
Therefore, 
\beaa
\Q(\T,\T,\T, L)&=&\vphi^3 \Q(L,L,L,L)+3\vphi^2 \Q(\Lb, L,L,L)\\
&+&3\vphi 
\Q(\Lb, \Lb,L,L)+Q(\Lb, \Lb, \Lb, L)\\
&=&\vphi^3|\a|^2+3\vphi^2|\b|^2+3\vphi(\rho^2+|\si|^2) +|\bb|^2
\eeaa
and the result follows from the bound 
 \eqref{eq:null-lapse-ineq} for $\vphi$.
\end{proof}
We can also get additional estimates for the flux associated to the
first derivatives of the curvature tensor. To see that we go back to 
the derivation of theorem  \ref{thm:deriv-energy-estim}.  We now integrate
\eqref{eq:divQ.wave} in $\JJ^-(p)$ and derive,
\beaa
\int_{\NN^{-}(p)} \Q^{(w)}[\R](\T, L)\le  \int_{\Si_{t_0} \cap \JJ^{-}(p) } \Q^{(w)}[\R](\T, L)+
\int_{\JJ^{-}(p)}|\D^\b(\Q^{(w)}[\R]_{\a\b}\T^\a)|.
\eeaa
Similarly,
\beaa
\int_{\NN^{-}(p,\de)} \Q^{(w)}[\R](\T, L)\le  \int_{\Si_{t(p)-\de} \cap \JJ^{-}(p,\de) } \Q^{(w)}[\R](\T, L)+
\int_{\JJ^{-}(p,\de)}|\D^\b(\Q^{(w)}[\R]_{\a\b}\T^\a)|.
\eeaa
Here $ \JJ^{-}(p)$ is the causal past of $p$ and ${\JJ^{-}(p,\de}$
the portion of  $ \JJ^{-}(p)$ to the future of $\Si_{t(p)-\de}$.
Now,
\beaa
\Q^{(w)}[\R](\T, L)&=&\Q^{(w)}[\R](\vphi L+\vphi^{-1}\Lb, L)
\eeaa
and,
\beaa
\Q^{(w)}[\R](L, L)&=& |\nabb_L\R|^2\\
\Q^{(w)}[\R](\Lb , L)&=&|\nabb \R|^2
\eeaa
We introduce the flux quantities,
\bea
\FF^{(1)}(p)&=&\int_{\NN^{-}(p)}
(|\nabb \R|^2+|\nabb_L\R|^2),\qquad \nn \\
\FF^{(1)}(p, \de)&=&\int_{\NN^{-}(p,\de)}(|\nabb \R|^2+|\nabb_L\R|^2)
\eea

We can therefore reformulate proposition  \ref{thm:deriv-energy-estim}
as follows,
\begin{theorem}
Assume that
{\bf A1}, {\bf A2} hold true. Then,
 for any $0<\de\le i_*$, with $i_*>0$ defined
by theorem \ref{thm:injectivity}, 
\bea
\|\D\R(t)\|_{L^2}   +\sup_{p\in\Si_t}\FF^{(1)}(p,\de)   & \le &C
 \bigg(\|\D\R(t-\de)\|_{L^2}+
 \big(\int_{t-\de}^{t}\|\R(t')\|_{L^\infty} dt'\big)^{1/2} \bigg)\nn\\\
 \label{estimate:FF1}
\eea
with $C$ a constant depending only on $\De_0$, $\RR_0$, $t_*$. 
\end{theorem}
 \subsection{Estimates for the Ricci coefficients}
 In this section we state without proof a proposition
 concerning the regularity properties of the  Ricci coefficients 
 $\trch$, $\chih$, $\ze$ and $\etab$  
 as well as mass aspect function $\mu$ 
 associated to the $S_t$ foliation. A similar result
 was proved for the corresponding quantities 
 associated with the geodesic foliation 
 in \cite{Causal1}, see also  \cite{Wang}.  The methods  
 used for the geodesic foliations can be easily adapted 
 to prove the result below.

\begin{proposition}\label{prop:coeff}
For any $t\in (t(p)-\de,t(p))$ with $\de<i_*^-(p,t)$ the Ricci coefficients 
$\trch$, $\chih$, $\ze,\etab$ and 
$\mu$ satisfy the following estimates.
\begin{align}
&\sup_{S_t} |\trch-\frac 2{s(t)}|+\|\sup_{t\in (t(p)-\de,t(p))} (t(p)-t) |\nabb\trch|\,\|_{L^2_\omega} \le C,\label{eq:trch}\\
&\sup_{\omega\in {\Bbb S}^2} \int_{t(p)-\de}^{t(p)}\left (|\chih|^2+|\zeta|^2+|\etab|^2)(t,\omega\right)\, dt\le C,\qquad
\|\mu\|_{L^2(\NN^-(p,\de))}\le C\label{eq:mu}
\end{align}
with a constant $C$ depending only on    $\De_0, t_*$  and 
 curvature flux $\RR(p,\de))$. Here the point $\NN^-(p,\de)$ are parametrized by the coordinates 
 $(t,\omega)$ with $\omega\in {\Bbb S}^2$. The volume forms $dA_{\NN^-(p,\de)}$ on $\NN^-(p,\de)$
and $dS_t$ on $S_t$ are respectively 
equivalent to the expressions $(t(p)-t)^2 \, dt\, d\sigma_{{\Bbb S}^2}$ and  
$(t(p)-t)^2 \, d\sigma_{{\Bbb S}^2}$ with $d\sigma_{{\Bbb S}^2}$ denoting the standard volume form
on ${\Bbb S}^2$. Notation $L^2_\omega$ above refers to the $L^2$ norm with respect to the measure
$d\sigma_{{\Bbb S}^2}$. Finally, the quantities $(t(p)-t)$ and $s(t)$ are equivalent.
\end{proposition}

\section{Kirchoff-Sobolev  Parametrix}

Earlier in this paper, see Propositions \ref{thm:deriv-energy-estim}, \ref{thm:deriv-energy-estim2}, 
we were able to derive $L^2$ estimates
for derivatives of the curvature tensor which depend on the additional
assumption on the boundedness of the  $L^\infty$ norm of the 
curvature tensor. To estimate the latter we rely on a special  version 
  of the Kirchoff-Sobolev  parametrix introduced in \cite{Kirch}. 
\subsection{Optical function} 
 To make sense of
our Kirchoff-Sobolev formula we need  to define an optical
function\footnote{ i.e. a function which verifies \eqref{eq:Eikonal} below.} $u$, 
in a neighborhood
of $\NN^{-}(p,\de)$, $0<\de<i_*^-(p,t) $,
  such that it  vanishes identically  on $\NN^{-}(p,\de)$.  Here $p$ is an arbitrary point of $\MM_*=\cup_{t\in [t_0, t_*)} \Si_t$. We recall that we have assumed 
  that $\MM_*$ is globally hyperbolic with Cauchy hypersurface $\Si_{t_0}$. 
 We  define $u$ 
uniquely  relative to  the  time-like vector $\T_p$ as 
follows:

 Let  $\ep>0$ a small number and  $\Ga_\ep:(1-\ep, 1+\ep) \to \MM_*$ denote  the
  timelike geodesic  from $p$  such that  $\Ga_\ep(1)=p$ 
 and $\Ga_\ep'(1)=\T_p$. 
From every point $q$  of $\Ga_\ep$ 
let $\NN^{-}(q)$  be the boundary of the past  set of $q$ and $\NN^{-}(q,\de)$
 defined as before.

We now  define $u$ to be the function, constant on each $\NN^{-}(q,\de)$,
such that for $q=\Ga(t)$,
$$u|_{\Null(q)}=t-1.$$
This defines a smooth function 
$u$ which vanishes on $\NN^{-}(p,\de)$  
and verifies the eikonal equation,
\be{eq:Eikonal}
\g^{\a\b}\pr_\a u\pr_\b u =0.
\end{equation}
 Observe that the vectorfield 
$L=\g^{\a\b}\pr_\b u\, \pr_\a$ is null, geodesic and 
verifies the normalization condition,
$$\g(L, \T_p)=\T_p(u)=1.$$
Thus $L$ is the same as the vectorfield $L$ defined earlier in section \ref{sec:past}.

\subsection{Main representation  formula}
We shall next state a result which was proved in \cite{Kirch}, concernig tensorial wave
equations of the form $\square\Psi =F$, with $\Psi$  a $k$-covariant tensorfield. 
Let  $p\in\MM_*$ 
 and  $\de< i_*^{-}(p,t)$.   Let   $\A$ be a tensor-field   of the same order   verifying,
 \be{eq:transp-A}
\D_L \A+\f12 \A \trch =0,\quad  s\A(p)=\J_0 \qquad{\text on}\quad \NN^-(p,\de)
\end{equation}
where $\J_0$ is a fixed k-tensor at $p$, $|\J_0(p)|\le 1$.

\begin{theorem} Let      $p\in\MM_*$ 
 and  $\de< i_*^{-}(p,t)$. Let  $\Psi$ be a $k$ covariant tensor-field
 vanishing identically  for $t\le t(p)-\de$. Then, given $\A$ a solution to the
 transport equations \eqref{eq:transp-A} we have\footnote{Here $\lapp$ denotes the angular  Laplace-Beltrami operator on the $2$-surfaces 
 $S_t$.},
\bea
4\pi \Psi(p)\c \J_0&=&-\int_{\NN^{-}(p,\,\de)}      \Big(   \A\,  \square \Psi
-\frac 12\A\c \R(\c,\c, \Lb\, , L)\c \Psi \Big)
+ \f12\int_{\NN^{-}(p,\,\de)}
\,\mu\,\A\c\Psi \nn\\
&+&\int_{\NN^{-}(p,\,\de)} \big(\lapp \A
+\ze_a \nabb_a \A \big)\, \c\Psi
\label{eq:main-ident}
\eea

\end{theorem}
We   apply   the theorem  to the tensor-field $\Psi= f(t)\R$
where $0\le f\le 1 $ is a smooth function supported in the
 interval $[t(p), t(p)-\de]$ and identically equal to $1$
 in the interval   $[t(p), t(p)-\de/2]$. 
 Since $\R$ verifies  \eqref{eq:wave1} we have,
\beaa
\square( f\R)= f \R\star \R +(\square f ) \R +2\D^\a f \D_\a \R
\eeaa
with $\R\star \R $ defined in \eqref{eq:notation1}. 
In view of  the theorem above we have the formula,
\bea
4\pi \R(p)\c \J_0&=&I(p)+J(p)+K(p) +L(p)+\EE\label{eq:IJKL}\\
I(p)&=& \int_{\NN^{-}(p,\,\de)}\A\c f(\R\star\R)\nn\\
J(p)&=&-\frac 12 \int_{\NN^{-}(p,\,\de)}\A\c \R(\c,\c, \Lb\, , L)\c f\R\nn\\
K(p)&=&\int_{\NN^{-}(p,\,\de)} \big(\lapp \A
+\ze_a \nabb_a \A \big)\, \c f \R\nn\\
L(p)&=&\f12\int_{\NN^{-}(p,\,\de)}\,\mu\,\A\c f\R\nn\\
\EE&=&\int_{\NN^{-}(p,\,\de)}\big(\,\square f (\A \c \R)+2  \D^\a f(  \A\c   \D_\a \R) \big)
\eea

\subsection{Estimates for  $I(p)$.}

We consider the orthonormal frame $E_0=\T, E_1, E_2, E_3$ which is well defined
everywhere in a neighborhood of the vertex $p$. Clearly the  norm $|U|$,
of an arbitrary  tensorfield $U$, defined according to definition \eqref{norm-U}
 coincides with  the square root of
the sum of squares of all the components of the tensor relative
to this orthonormal frame. It is easy to se that,
$$|\A \c (\R\star \R)|\le |\A|\c |\R\star \R|.$$

On the other hand,  if $e_4=L, e_3=\Lb$ the null pair \eqref{eq:null frame}
 and  denote by $\aa(\R\star \R), \bb(\R\star \R), \rho(\R\star \R), \si(\R\star
\R), 
\b(\R\star \R), \aa(\R\star \R)$ the null decomposition
of $\R\star \R$, as  a Weyl field, relative to  the null pair 
$ e_3, e_4$  we  can easily check that,
\bea
|\R\star\R|^2&\les &|\aa(\R\star \R)|^2+|\bb(\R\star \R)|^2+|\rho(\R\star
\R)|^2\label{eq:RstarR-null}\\
 &+&|\si(\R\star \R)|^2+|\b(\R\star \R)|^2+|\a(\R\star \R)|^2
\eea
Indeed if we denote the Weyl field $\R\star \R$ by $W$ and introduce
its electric and magnetic parts $E_{ij}=W_{i0j0}$, $H_{ij}=^\star W_{i0j0}$
we have,
$$ |W|^2=|E|^2+|H|^2.$$
Indeed, we have,
$$\begin{array}{ll}
W_{ijk0}=-\in_{ij}^{\,\, s} H_{sk},&\qquad ^\star W_{ijk0}=\in_{ij}^{\,\, s} E_{sk}\\
W_{ijkl}= -\in_{ijs}\in_{klt} E^{st}, &\qquad ^\star W_{ijkl}= -\in_{ijs}\in_{klt} H^{st}
\end{array}
$$
On the other hand, in terms of the null decomposition of 
$W$, relative to $ e_1, e_2, e_3, e_4$,
$$
\begin{array}{ll}
E_{ab}=\frac 1 4 \a_{ab}+\frac 1 4 \aa_{ab}-\f12 \rho \de_{ab} &\quad H_{ab}=-\frac 1 4
^\star\a_{ab}+\frac 1 4 ^\star\aa_{ab}-\f12 \si \de_{ab}\\
E_{aN}=\f12 \,\,\bb_a+\f12 \b_a &\quad H_{aN}=\f12\,\, ^\star\bb_a- \f12\,\,
^\star\b_a\\ E_{NN}=\rho&\quad H_{NN}=\si
\end{array}
$$
Hence,
$$|W|^2=|E|^2+|H|^2\les |\a|^2+|\b|^2+|\rho|^2+|\si|^2 +|\bb|^2+|\aa|^2$$
which proves \eqref{eq:RstarR-null}.
We now estimate the right hand side of  \eqref{eq:RstarR-null}.
Clearly any null component of  $\R\star \R$ can be expressed
as a quadratic expression
in the null components of 
$\R$. We  observe  that no null component of $\R\star\R$
 can be quadratic in $\aa$. This can be easily proved by 
a signature consideration. Indeed  we assign signature
$2$ to $\a(\R)$, signature 1 to $\b(\R)$ signature $0$ to $\rho(\R)$ and $\si(\R)$,
signature $-1$ to $\bb(\RR)$ and signature $-2$ to $\aa(\R)$. 
Similarily  we assign  signature $2$ to $\a(\R\star \R)$, signature 1 to 
$\b(\R\star \R)$ signature $0$ to
$\rho(\R\star \R)$ and $\si(\R\star \R)$, signature $-1$ to $\bb(\RR\star \R)$ 
and signature $-2$ to
$\aa(\R\star \R)$. It is easy to check that in the algebraic formula expresses  the
null components of $\R\star \R$ in terms of a quadratic form  
in the null components of $\R$
the total signature of each term 
must be the same as  the signature
of the corresponding null component of 
$\R\star \R$. Thus,
\beaa
\aa(\R\star \R)&=&(\R\star\R)_{a3b3}=\Qr[\aa, (\rho,\si)]+\Qr[\bb, \bb]
\eeaa
where $\Qr[\,, \,] $ denotes a  simple quadratic expression  in the null components
of $\RR$. Similarily,
\beaa
\bb(\R\star \R)&=&\Qr[\aa, \b]+\Qr[\bb, (\rho,\si)]\\
\rho(\R),\si(\R)&=&\Qr[\aa, \a
]+\Qr[\bb, \b]+\Qr[(\rho,\si), (\rho,\si)]\\
\b(\R\star \R)&=&\Qr[\bb, \a]+\Qr[ (\rho,\si), \b])\\
\a(\R\star \R)&=&\Qr[(\rho,\si), \a]+\Qr[ \b, \b ])
\eeaa
We now introduce the notation 
\be{eq:funny-norm}
(|\R|^\dagger)^2= |\a|^2+|\b|^2+|\rho|^2+|\si|^2 +|\bb|^2
\end{equation}
and deduce the following inequality,
\be{eq:important-pointwise}
|\R\star\R|\les |\R |^\dagger\c|\R|
\end{equation}
 
We are now ready to estimate the term $I(p)$. Using
the bounds for  $ \RR(p)$  of theorem \ref{thm:geom-estim} and 
 $\RR(p,\de)\le \RR(p)$ we derive,
 \beaa
 |I(p)|&\les & \int_{\NN^{-}(p,\,\de)}|\A|  |\R |^\dagger\c |\R| \\
 &\les&\big(\int_{\NN^{-}(p,\,\de)} |\R |^\dagger \big)^{1/2}  \big(\int_{\NN^{-}(p,\,\de)}  |\A|^2 |\R|^2\big)^{1/2}\\
 &\les& \RR(p,\de)
 \big( \int_{t(p)-\de}^{t(p)}\|\R(t)\|^2_{L^\infty}\|\A(t)\|^2_{L^2(S_t)} dt\big)^{1/2}\\
  &\les& \RR(p,\de) \big(\int_{t(p)-\de}^{t(p)}\|\A(t)\|^2_{L^2(S_t)}\big)^{1/2} \sup_{t(p)-\de<t< t(p)}
  \|\R(t)\|_{L^\infty(S_t)}
 \eeaa
 We therefore have
 \bea
 |I(p)|&\les & \RR(p,\de) \,   \|\A\|_{L^2(\NN^-(p,\de))}  \, \|  \R\|_{L^\infty(\NN^-(p,\de))} 
 \label{eq:main-error1}
 \eea

\subsection{Estimate for $J(p)$.}
 In view of  the fact that, 
$$\int_{\NN^-(p, \de)} | \R(\c,\c, \Lb\, , L)|^2\les \RR(p,\de)^2,
$$
we deduce, proceeding exctly as for $I$,

Therefore,
\be{eq:main-error2}
|J(p)|\les  \RR(p,\de) \,   \|\A\|_{L^2(\NN^-(p,\de))}  \, \|  \R\|_{L^\infty(\NN^-(p,\de))} 
\end{equation}

\subsection{Estimates for $L(p)$.}
We proceed as follows,
\bea
|L(p)|&\les& \int_{\NN^{-}(p,\,\de)}\,|\mu|\, |\A|\,|\R|
\les \big( \int_{\NN^{-}(p,\,\de)}\,|\mu|\,\big)^{1/2}
\big( \int_{\NN^{-}(p,\,\de)}\, |\A|^2\,|\R|^2\,\big)^{1/2}\nn\\
&\le& \|\mu\|_{L^2(\NN^{-}(p,\de))}  \,   \|\A\|_{L^2(\NN^-(p,\de))}  \, \|  \R\|_{L^\infty(\NN^-(p,\de))} 
\label{eq:main-error3}
\eea

\subsection{The term $K(p)$}
Integrating by parts we  rewrite 
$K(p)$  as follows,
\bea
K(p)&=&- \int_{\NN^{-}(p,\,\de)} \nabb  \A\cdot \nabb \R
+  \int_{\NN^{-}(p,\,\de)}\ze_a \D_a \A \, \c\R\nn\
\eea
We now estimate as follows,
\beaa
|K(p)|&\les&\int_{\NN^{-}(p,\,\de)} \big(|\nabb  \A||\nabb\R|
+|\ze| |\nabb \A |\, |\R|\big)\\
&\les& \|\nabb A\|_{L^2(\NN^{-}(p,\,\de))} \big( \|\nabb \R\|_{L^2(\NN^{-}(p,\,\de))}
+\|\ze \R\|_{L^2(\NN^{-}(p,\,\de))}\big)
\eeaa
Therefore,
\bea
|K(p)|&\les&\|\nabb \A\|_{L^2(\NN^{-}(p,\,\de))}\big( \|\nabb \R\|_{L^2(\NN^{-}(p,\,\de))}+\|\R\|_{L^\infty (\NN^{-}(p,\,\de))}\|\ze\|_{L^2(\NN^{-}(p,\,\de))}\big)\nn
\eea

 Going back to
\eqref{eq:IJKL} and 
using  the estimates for $I,J,K,L$  obttained above we derive,
\begin{proposition} The following estimate holds  for all $p \in \MM_*$
 and $0<\de < i_*^{-}(p,\de)$.
\beaa
|\R(p)|&\les&\EE+
 \|  \R\|_{L^\infty(\NN^-(p,\de))}
  \|  \A\|_{L^2(\NN^-(p,\de))} \big(\RR(p,\de)+\|\mu\|_{L^2(\NN^{-}(p,\de))}\big)\\
&+&\|\nabb \A\|_{L^2(\NN^{-}(p,\,\de))}\big( \|\nabb \R\|_{L^2(\NN^{-}(p,\,\de))}
+ \|  \R\|_{L^\infty(\NN^-(p,\de))} \|  \ze \|_{L^2(\NN^-(p,\de))}\big) 
\eeaa
\end{proposition}
We now recall that according to \eqref{eq:Rp}
$$
\RR(p,\de)\les \RR(p)\le C(t_*,\De_0) \RR_0.
$$
Furthermore, by \eqref{eq:mu}
$$
\|\mu\|_{L^2(\NN^{-}(p,\de))}\le C(t_*,\De_0,\RR_0) 
$$
and 
\beaa
\|  \ze \|_{L^2(\NN^-(p,\de))}&=&\left (\int_{t(p)-\de}^{t(p)} n dt \int_{S_t} |\ze|^2 dA_t\right)^{\frac 12}\\
&\le& \de \sup_{\omega_{\Bbb S}^2} \left(\int_{t(p)-\de}^{t(p)} |\ze(t,\omega)|^2\, dt\right)^{\frac 12}\\&\le&
\de C(t_*,\De_0,\RR_0).
\eeaa
Finally, in the next section we will establish that
$$
\|( t(p)-t) \A\|_{L^\infty(\NN^-(p,\de))}\le C,\qquad \|\nabb \A\|_{L^2(\NN^{-}(p,\,\de))}\le C
$$
with a constant $C=C(t_*,\De_0,\RR_0)$. This in particular implies that
$$
\|  \A\|_{L^2(\NN^-(p,\de))}\le \de^{\frac 12} C(t_*,\De_0,\RR_0).
$$
Putting this all together we deduce,
\beaa
\|\R(t)\|_{L^\infty}&\les&\EE+\de^{\frac 12} \sup_{p\in \Si_t} \|  \R\|_{L^\infty(\NN^-(p,\de))}
+\sup_{p\in \Si_t} \|\nabb \R\|_{L^2(\NN^{-}(p,\,\de))}
\eeaa
On the other hand, according to \eqref{estimate:FF1}
we have,
\beaa
\sup_{p\in \Si_t} \|\nabb \R\|_{L^2(\NN^{-}(p,\,\de))}=\sup_{p\in\Si_t}\FF^{(1)}(p,\de)\les 
 \|\D\R(t-\de)\|_{L^2}+ \big(\int_{t-\de}^{t}\|\R(t')\|^2_{L^\infty} dt'\big)^{1/2} 
\eeaa
Therefore,
\begin{equation}
\|\R(t)\|_{L^\infty}\les\EE+\|\D\R(t-\de)\|_{L^2}+ \de^{\frac 12}
\sup_{t'\in(t-\de,t)}\|\R(t')\|_{L^\infty}\label{pre-final}
\end{equation}
\subsection{Estimates for the error term $\EE$}
We  first observe that,
\beaa
|\square t|\les 1,\qquad  |\D t|\les 1.
\eeaa
Therefore, since $f'$ and $f''$ vanish for $|t-t(p)|\le \de/2$
and 
\beaa
\|f'\|_{L^\infty} \le \de^{-1}, \qquad  \|f''\|_{L^\infty} \le \de^{-2}
\eeaa
we derive,
\beaa
|\EE|&\les&\de^{\frac 12}
\|\A\|_{L^2(\NN^{-}(p,\de)} \sup_{t'\in[t-\de, t-\de/2]}
\big(\de^{-2}   \|\R(t')\|_{L^2(S_t')}dt+ \de^{-1}       \|\D \R(t')\|_{L^2(S_t')} \big)dt'\\
&\les&\sup_{t'\in[t-\de, t-\de/2]}\big(\de^{-1}   \|\R(t')\|_{L^2(S_t')}dt+       \|\D \R(t')\|_{L^2(S_t')} \big)
\\
&\les& \sup_{t'\in[t-\de, t-\de/2]}\big(\de^{-1} \left (\|\R(t')\|_{L^2(\Si_{t'})}+\|\D\R(t')\|_{L^2(\Si_{t'})}\right)
+ \|\D^2\R(t')\|_{L^2(\Si_{t'})}\big)
\eeaa
The last step can be justified by a simple  integration by parts argument.
\subsection{Final Estimate}
Returning to \eqref{pre-final}, taking a supremum in $t$ over an interval of size $\de$ 
and using the Sobolev  inequality  of corollary \ref{cor:Sobolev},
$$
\|\R(t')\|_{L^\infty}\le C\left( \|\R(t')\|_{L^2}+\|\D\R(t')\|_{L^2}+ \|\D^2\R(t')\|_{L^2}\right),
$$
we obtain,
\begin{proposition}
\label{prop:funny}
There exists  a positive $\de>0$, sufficiently small but depending only on $\De_0, \RR_0, t_*$,  such that 
the following estimate holds true,
\beaa
\|\R(t)\|_{L^\infty}& \le& C\de^{-1} \sup_{t'\in[t-2\de, t-\de/2]}\left( \|\R(t')\|_{L^2}+\|\D\R(t')\|_{L^2}+ \|\D^2\R(t')\|_{L^2}\right)
\eeaa
with $C$ a constant depending only on $\De_0$, $t_*$, $\RR_0$.
\end{proposition}
We now return to  propositions  \ref{thm:deriv-energy-estim} and \ref{thm:deriv-energy-estim2}. Combining them  with 
 the proposition above we deduce, 
 \beaa
\|\D\R(t)\|_{L^2}^2&\le& C \big(\|\D\R(t-\de/2)\|_{L^2}^2+\int_{t-\de/2}^t \|\R(t')\|_{L^\infty}^2  dt' \big)\\
&\le& C\de^{-1} \sup_{t'\in[t-\de, t-\de/2]}\left( \|\R(t')\|_{L^2}+ \|\D\R(t')\|_{L^2}^2+ 
\|\D^2\R(t')\|_{L^2}^2\right)\\
\|\D^2\R(t)\|_{L^2}^2 &\le& C\big(  \|\D^2\R(t-\de)\|_{L^2}^2+ \int_{t-\de/2}^t\|\D\R(t')\|^2_{L^2}\|\R(t')\|_{L^\infty}^2 dt'\big)\\
&\le&  C  \|\D^2\R(t-\de)\|_{L^2}^2\\&+& C\de^{-2} \sup_{t'\in[t-\de, t-\de/2]}
\left( \|\R(t')\|_{L^2}+ \|\D\R(t')\|_{L^2}^2+ \|\D^2\R(t')\|_{L^2}^2\right)^2
\eeaa
Consequently, for some $C$ depending only on $\De_0, \RR_0$ and $t_*$,
\bea
\|\R(t)\|_{H^2}\les C \de^{-1}  \sup_{t'\in[t-\de, t-\de/2]}\|\R(t')\|_{H^2}^2
\eea
where,
\bea
\|\R(t)\|_{H^2}=\|\R(t)\|_{L^2}+\|\D\R(t)\|_{L^2}+\|\D^2\R(t)\|_{L^2}
\eea
Iterating the estimate as many times as needed, in steps of size $\de/2$,  we derive,
\begin{theorem}
\label{mainthm-energy}
Assume that $(\M,\g)$
is a globally hyperbolic extension of $\Si_0$ verifying  the assumptions {\bf A1} and {\bf A2}.
Let $\MM_*=\cup_{t\in[ [-t_0, t_*)} \Si_t \subset \M$ with $t_0=-1$. 
There exists a constant $C>0$ depending   only on $\De_0, t_* $ and initial data  $\|\R(t_0)\|_{H^2}$
such that,
\bea
\sup_{t\in [t_0, t_*)}\|\R(t)\|_{H^2}\le C
\eea 

\end{theorem}
\section{Proof of  Main Theorem \ref{thm:main}}
Theorem \ref{mainthm-energy} established above provides
us with  global uniform bounds  for the curvature tensor $\R$
and $L^2$ bounds for  its first two covariant derivatives.  Using elliptic estimates  we can also derive $L^2$  bounds for the first three derivatives 
of the second fundamental form, see theorem \ref{thm-last}. 
To finish the  proof of the Main Theorem we only need to  apply  the following local existence result
\begin{proposition}\label{prop:loc}
Let $(\Si_*,g,k)$ be initial data for the Einstein vacuum equations satisfying  the constraint equations.
We assume that $\Si$ is compact and has constant   mean curvature $\tau=g^{ij} k_{ij}={\text{const}}<0$.  
Let $R$ denote the Ricci  curvature tensor of $g$. Then there exists a smooth 
future Cauchy development of $(\Si,g,k)$ containing the region $\cup_{t\in [\tau,\tau+\rho]}\Si_t$,
where each $\Si_t$ is a constant mean curvature hypersurface (with mean curvature equal to $t$)
and $\Si_\tau=\Si_*$. The  constant $\rho$  here depends only on the diameter and radius of injectivity 
of $\Si$, the strictly negative   constant  $\tau$ ($\tau<0$)  and
 the  following  constant,
\bea\label{eq:sigma}
\RR_*&=&\|R\|_{L^2(\Si)} + \|\nab R\|_{L^2(\Si)}+ \|\nab^2 R\|_{L^2(\Si)}  \\
&+& \|k\|_{L^4(\Si)} +\|\nab k\|_{L^2(\Si)}+ \|\nab^2 k\|_{L^2(\Si)}\nn
+ \|\nab^3 k\|_{L^2(\Si)}.
\eea
\end{proposition}
\begin{proof}: The proof requires a slight modification of  the local exitence theorem 10.2.1 in \cite{C-K}.
\end{proof}
Theorem \ref{mainthm-energy} above, combined with the bounds on the second fundamental form stated in theorem \ref{thm-last}, proved
below in the Appendix, implies that for each hypersurface $\Si_t\subset\MM_*$ with 
$t_0\le t<t_*$ the  constant $\RR_*$, defined in \eqref{eq:sigma}, is uniformly bounded.  On the other
hand, Theorem \ref{thm:harmonic} together with $L^\infty$ bounds on curvature implied by 
Theorem \ref{mainthm-energy} guarantees a uniform bound for both the diameter and radius of 
injectivity of $\Si_t$ for $t_0\le t<t_*$. As a consequence, under assumptions {\bf A1}, {\bf A2},
as long as $t_*<0$ we can construct a smooth globally hyperbolic CMC development containing 
the region $\cup_{t\in [t_0,t_*]} \Si_t$.

\section{Estimates for $\A$.}

\begin{proposition} Let $\A$ be the tensor defined in \eqref{eq:transp-A}.
Then for all $0<\de<i_*^-(p,t)$,
\be{eq:estim-A-unif}
\|(t(p)-t) \A\|_{L^\infty(\NN^-(p,\de))}\le C(t_*,\De_0,\RR_0).
\end{equation}

\end{proposition}
\begin{proof}.\quad 
Recall, see definition \ref{def:notation}, the convention $b\les 1$ for an inequality of 
the form $b\le C(t_*,\De_0,\RR_0)$. 
We  claim that it suffices to prove the proposittion
for the case when $\A_\mu$ is a vectorfield. The general case can be 
derived by a simple induction argument. Recall that we have,
$$
\D_L \A + \frac 12 \trch \A=0.
$$
with $ (s\A)$  prescribed to be $\J_0$ at the vertex $p$. 
In  view 
of the identity 
$
\frac {dt}{ds}=-(n\varphi)^{-1}
$
 as well as  estimate \eqref{eq:null-lapse-ineq}, it suffices to  prove the inequality 
$$
\|s\A\|_{L^\infty(\NN^-(p,\de))}\les 1.
$$
Letting  $\B=s\A$  we have
\bea
\D_L \B=-1/2 (\frac{2}{s}-\trch)\B,\qquad \B|_{s=0}=\J_0
\eea
Recall that,
\beaa
|\B|^2=B_0^2+|\underline{B}|^2=2|B_0|^2+<\B,\B>
\eeaa
where $B_0=<\B,\T>$, $\underline{B}$ is the projection of $\B$ on
the foliation $\Si_t$ and $<\B,\B>=\g^{\mu\nu} \B_\mu\B_\nu$.
 We shall first estimate $<\B,\B>$
by observing that,
\beaa
\frac{d}{ds}<\B,\B>=- (\frac{2}{s}-\trch)<\B,\B>,
\eeaa
which in turn implies 
\beaa
\frac{d}{dt}<\B,\B>= n\varphi (\frac{2}{s}-\trch)<\B,\B>,
\eeaa
Therefore, since  $|<\B,\B>(0)|\les |\J(0)|\les 1$,
\beaa
 |<\B(t),\B(t)>|\le \,  |<\B(0),\B(0)>|
\exp\big(\int_t^{t(p)}n\varphi \,|\trch-\frac{2}{s'}| dt'\big)\les 1,
\eeaa
where the last inequality follows from \eqref{eq:null-lapse-ineq} and \eqref{eq:trch}.
Therefore, for all $t(p)-\de< t\le t(p)$,
\be{eq:estim-Delta}
|<\B(t),\B(t)>|\les 1.
\end{equation}
We shall next derive a transport equation for
$B_0$ using the fact that $\B=-B_0 \T  +\Bund$
\beaa
\frac{d}{ds} B_0&=&<\D_\L \B, \T>+<\B, \D_\L\T>\\
&=& 1/2 (\frac{2}{s}-\trch)B_0+<\Bund , \D_\L \T>
\eeaa
Observe that,
\beaa
<\Bund , \D_\L \T>=-\f12 \varphi^{-1}\big(<\Bund, \D_\T\T+\D_N \T>\big)
\eeaa
Therefore, recalling our condition \eqref{bound.DT},
\beaa
|<\Bund, \D_T\T+\D_N \T>|\les \De_0  |\Bund|
\eeaa
Therefore,
\beaa
\frac{d}{ds}|B_0|\les |B_0|+\De_0|\Bund|
\eeaa
On the other hand, from \eqref{eq:estim-Delta},
$|-B_0^2+|\Bund|^2|=|<\Bund,\Bund>|\les 1$ from which,
\beaa
|\Bund|\les \sqrt{1+B_0^2}\les |B_0|+1
\eeaa
Therefore,
\beaa
\frac{d}{dt}|B_0|\les |B_0|+\De_0(|B_0|+1)
\eeaa
from which we deduce  the estimate
\beaa
|B_0|\les 1
\eeaa
Thus, together with \eqref{eq:estim-Delta},
we derive,
\be{eq:estim-B0}
|\B|\le |B_0|+|\Bund|\les 1
\end{equation}
as desired.
\end{proof}
\subsection{Estimates for $\nabb \A$.}

\begin{proposition}
Let $\A$ be the tensor defined in  \eqref{eq:transp-A}.
Then,
\be{eq:estim-nabA}
\|\sup_{t(p)-\de\le t\le t(p)} (t(p)-t)^{\frac 32} |\nabb\A(t)|\, \|_{L^2_\omega} \les 1,
\end{equation}
\be{eq:estim-nabA-L2}
\|\nabb\A\, \|_{L^2(\NN^-(p,\de))} \les 1,
\end{equation}
\end{proposition}
\begin{proof}
In what follows we recall that points on $\NN^-(p,\de)$ are parametrized by the coordinates 
 $(t,\omega)$ with $\omega\in {\Bbb S}^2$. According to to Proposition \ref{prop:coeff}
 the volume forms $dA_{\NN^-(p,\de)}$ on $\NN^-(p,\de)$
and $dS_t$ on $S_t$ are respectively 
equivalent to the expressions $(t(p)-t)^2 \, dt\, d\sigma_{{\Bbb S}^2}$ and  
$(t(p)-t)^2 \, d\sigma_{{\Bbb S}^2}$ with $d\sigma_{{\Bbb S}^2}$ denoting the standard volume form
on ${\Bbb S}^2$. Similarly equivalent are the quantities $(t(p)-t)$ and $s(t)$.
Notation $L^2_\omega$ above refers to the $L^2$ norm with respect to the measure
$d\sigma_{{\Bbb S}^2}$. 

We begin by applying the results of lemma \ref{lem:commute}
to the equation $
\D_\L \A + \frac 12 \trch \A=0.
$ and derive,
\beaa
\nabb_\L (\nabb_a A_\mu) +\chi_{ab} \nabb_b A_\mu= -\f12 \nabb_a (\trch \A_\mu) -\f12(\ze_a+\etab_a)
\trch \A_\mu+
 {\R_{\mu}^{\,\,\la}}_{L a} A_\la
\eeaa
Therefore,
\beaa
\nabb_\L (\nabb_a \A_\mu)+\trch (\nabb_a \A_\mu)=-\chih_{ab}\nabb_b \A_\mu
-\f12 (\nabb_a \trch) \A_\mu  -\f12(\ze_a+\etab_a)
\trch \A_\mu+ 
 {\R_{\mu\,\,L a}^{\,\,\la}} A_\la
\eeaa
which we rewrite in the form
\bea
\nabb_\L \U_{a\mu}+\trch\U_{a\mu}& =&-\chih_{ab}\U_{b\mu}+\F_{a\mu},\qquad \U|_{s=0}=0\\
\F_{a\mu}&=&
-\f12 (\nabb_a \trch) \A_\mu  -\f12(\ze_a+\etab_a)
\trch \A_\mu+ 
 {\R_{\mu\,\,\L a}^{\,\,\la}} A_\la
\eea
with $\U_{a\mu}=\nabb_a \A_\mu$. Observe that
$$
|\U|^2=2 |\U_{a0}|^2+\U_{a\mu} \U_{a}^{\,\,\mu}
$$
Then,
$$
\nabb_\L \U_{a0}+\trch\U_{a0} =-\chih_{ab}\U_{b0}+\F_{a0}+\frac 12 \phi^{-1}(n^{-1}\nabb_j n + k_{N j}) \U_{a j}
$$As a consequence,
\beaa
\nabb_\L \left (s^2\U_{a0}\right)=s^2\left ((\frac{2}{s}-\trch-\chih)\c\U_{a0}+\F_{a0}+
\frac 12 \phi^{-1}(n^{-1}\nabb_j n + k_{N j}) \U_{a j}\right) := s^2 \G
\eeaa
Using that $\frac {ds}{dt}=-n\varphi$ and 
$$
\frac d{ds} \left (s^4 |\U_{a0}|^2\right) = s^4 \g(\G,\U_{a0})
$$
together with boundedness of $n$ and $\varphi$,
we estimate in the range  $t(p)-\de\le t\le t(p)$,
\begin{align*}
\|\sup_{t} s(t)^3  |\U_{a0}|^2\|_{L^1_\omega} &\les 
\|\sup_t s(t)^{-1} \int_{t}^{t(p)} s(\tau)^4 \g( \G,\U_{a0})\, d\tau  \|_{L^1_\omega}\\&
\les \|\sup_t s(t)^{-1} \int_t^{t(p)} s(\tau)^{\frac 52} |\G|d\tau \|_{L^2_\omega} \| \sup_{\tau} s(\tau)^3  |\U_{a0} |^2 \|_{L^1_\omega}^{\frac 12} \\ &\les \epsilon^{-1}
\|\sup_ts(t)^{-1}\int_t^{t(p)} s(\tau)^{\frac 52} |\G|d\tau \|^2_{L^2_\omega} + 
\epsilon
\| \sup_\tau s(\tau)^3  |\U_{a0} |^2 \|_{L^1_\omega}
\end{align*}
To control $\|\sup_t s(t)^{-1} \int_t^{t(p)} s(\tau)^{\frac 52} |\G|d\tau \|_{L^2_\omega}$ we first estimate the integral,
$$J:=\|\sup_t s(t)^{-1}
\int_t^{t(p)} s(\tau)^{\frac 52} |(\frac{2}{s(\tau)}-\trch-\chih)\c\U_{a0}|d\tau \|_{L^2_\omega},$$
as follows,
\begin{align*}
J& \les 
\|\sup_t \int_t^{t(p)} |(\frac{2}{s(\tau)}-\trch-\chih)|d\tau \|_{L^\infty_\omega} \| \sup_\tau s(\tau)^{\frac 32}  |\U_{a0}| \|_{L^2_\omega}
\\ &\les  \|\sup_t s(t)^{\frac 12}
\int_t^{t(p)} |(\frac{2}{s(\tau)}-\trch-\chih)|^2 d\tau \|^{\frac 12}_{L^\infty_\omega} 
\| \sup_\tau s(\tau)^{\frac 32}  |\U_{a0} | \|_{L^2_\omega}\\ &\les \de^{\frac 12} 
\| \sup_\tau s(\tau)^{\frac 32}  |\U_{a0}| \|_{L^2_\omega},
\end{align*}
where the last inequality follows from \eqref{eq:trch} and \eqref{eq:mu}.
On the other hand, in view of \eqref{eq:estim-A-unif}, \eqref{eq:trch},\eqref{eq:mu}
\begin{align*}
\|\sup_t s(t)^{-1}
\int_t^{t(p)} s(\tau)^{\frac 52} |\nabb\trch| |\A| d\tau \|_{L^2_\omega}& \les 
\de^{\frac 12}\|\sup_\tau (t(p)-\tau)  |\nabb\trch| \|_{L^2_\omega}\les \de^{\frac12},\\
\|\sup_t s(t)^{-1}
\int_t^{t(p)} s(\tau)^{\frac 52} |\ze+\etab| |\trch| |\A| d\tau \|_{L^2_\omega}& \les 
\|\sup_\tau (t(p)-\tau)  |\trch| \|_{L^\infty} \int_t^{t(p)} |\ze+\etab|^2d\tau d\omega\les 1.
\end{align*}
Moreover, using Proposition \ref{thm:geom-estim}, we have
\begin{align*}
\|\sup_t s(t)^{-1}\int_t^{(p)} s(\tau)^{\frac 52} |\R_{\mu\,\,\,L a}^{\,\,\la}| |\A| d\tau \|_{L^2_\omega}& \les 
\|\sup_t s(t)^{-\frac 12}
\int_t^{t(p)} (t(p)-\tau) |\R_{\mu\,\,\,L a}^{\,\,\la}| d\tau \|_{L^2_\omega}
\\ &\les \|\R_{\mu\,\,\,L a}^{\,\,\la}\|^{\frac 12}_{L^2(\NN^-(p,\de_*))} 
\les  {\mathcal{R}}_0^{\frac 12}.
\end{align*}
Note that it is the presence of an $L$ component in the Riemann curvature tensor 
$\R_{\mu\,\,\,L a}^{\,\,\la}$ which allows us to express it as a linear combination of 
the tangential terms $\a,\b,\rho,\sigma,\bb$ entering into the expression for the curvature flux.

Finally,
\begin{align*}
\|\sup_t s(t)^{-1}\int_t^{t(p)} s(\tau)^{\frac 52}  |\varphi^{-1}(n^{-1}\nabb_j n + k_{N j}) \U_{a j}| d\tau \|_{L^2_\omega}& \les 
\Delta_0 \de 
\| \sup_\tau s(\tau)^{\frac 32}  |\U_{a\cdot} | \|_{L^2_\omega}
\end{align*}
Therefore,
$$
\|\sup_t s(t)^3  |\U_{a0}|^2\|_{L^1_\omega} \les \epsilon^{-1} \left (1+s(t)^{2}+{\mathcal{R}}_0\right)+
(\epsilon+\epsilon^{-1} \Delta_0^2 s(t)^2) \| \sup_\tau s(t)^{3}  |\U_{a\cdot} |^2 \|_{L^1_\omega}
$$
Combining this with the similar estimate on $\U_{a\mu} \U_{a}^{\,\,\mu}$ we obtain 
$$
\| \sup_{t(p)-\de\le t\le t(p)} (t(p)-t)^3  |\U_{a0}|^2\|_{L^1_\omega} \les 1+{\mathcal{R}}_0,
$$
which gives \eqref{eq:estim-nabA}. The argument above also provides the inequality
$$
(t(p)-t)  \|\U\|_{L^2_\omega}\les (t(p)-t)^{\frac 12} + \frac 1{t(p)-t} \int_t^{t(p)} 
\left ((t(p)-\tau) \|\R_{\mu\,\,\,\L a}^{\,\,\la}\|_{L^2_\omega}+\|\ze+\etab\|_{L^2_\omega}\right)\,d\tau.
$$
Using the maximal function estimate we then obtain
$$
\|\U\|_{L^2(\NN^-(p,\de_t))}\les 1+ {\mathcal{R}}_0^{\frac 12}
$$
and hence \eqref{eq:estim-nabA-L2}.
\end{proof}
\section{Appendix}
Recall that the curvature tensor $\R$ can be decomposed
into its  electric and magnetic parts $E, H$ as follows,
\be{eq:el-magn}
E(X,Y)=<\R(X,\T)\T, Y>,\qquad H(X,Y)=<\dual\R(X,\T)\T, Y>
\end{equation}
with $\dual\R$ the Hodge dual of $\R$. One can easily check
that $E$ and $H$ are tangent, traceless 2-tensors,   to $\Si_t$ and 
that $|\R|^2=|E|^2+|H|^2$. We easily check the formulas relative
to an orthonormal frame $e_0=T, e_1, e_2, e_3$,
\bea
\R_{abc0}&=&-\in_{abs}H_{sc},\qquad \dual
\R_{abc0}=\in_{abs}E_{sc}\label{eq:R-coeff}\\
\R_{abcd}&=&\in_{abs}\in_{cdt}E_{st},\qquad 
\dual\R_{abcd}=-\in_{abs}\in_{cdt}H_{st}\nn
\eea

We recall below some of  the main 
formulas  involving $k, E$ and $H$.
\bea
E_{ij}-R_{ij}&=&\tr k \,k_{ij}-k_{i}^{\, s}k_{sj}\label{eq:E}\\
H_{ij}&=&\curl k_{ij}\label{eq:H}
\eea
where, for any given symmetric two tensor 
$l$ of $\Si_t$ one defines 
$$\curl l_{ij}=\in_{i}^{\,ab}\nab_a l_{bj}+ \in_{j}^{\,ab}\nab_a l_{ib}.$$
We also recall the constraint equation for $k$,
\be{eq:constr-div}
\nab^j k_{ij}-\nab_i\tr k=0
\end{equation}
In the particular case when $\tr k$ is constant
 equations \eqref{eq:H}
and \eqref{eq:constr-div} form an elliptic Hodge system on $\Si_t$,
\bea
\div k&=&0\qquad 
\curl k=H \label{eq:Hodge-k}
\eea

\subsection{Elliptic $L^2$- estimates for Hodge systems}
Here  we recall the  following lemma concerning rank-2 symmetric Hodge systems 
on a $3$ dimensional compact  Riemannian manifold $\Si$.
\begin{lemma} The following $L^2$ elliptic estimates hold
on a $3$ dimensional Riemannian manoflod $\Si$.
\label{le:elliptic-Sigma}

{\bf i.}\quad  Let $V$ be a symmetric tracelss $2-$ tensor 
on $\Si$ verifying,
\bea
\div V&=&\rho,\qquad 
\curl V=\si
\eea
Then,
\bea
\int_\Si\big(|\nab V|^2+3 R_{mn} V^{im} V_{i}^{\,n}-\f12 R
|V|^2\big)=\int_\Si(|\si|^2+\f12 |\rho|^2)
\eea
where $R_{ij}$  is the Ricci curvature of $\Si$ and $R$ its scalar
curvature.

{\bf ii.}\quad For a scalar $\phi$ we have,
\beaa
\int_\Si|\nab^2 \phi|^2+\int_\Si R^{ij}\nab_i\phi\nab_j \phi=
\int_\Si|\lap\phi|^2
\eeaa
\end{lemma}
\begin{proof}. \quad See Proposition 4.4.1 in \cite{C-K}.
\end{proof}
\subsection{Apriori estimates for $k$}We now  apply lemma
\ref{le:elliptic-Sigma} to the Hodge system \eqref{eq:Hodge-k} for $k$ on a
fixed  hypersurface $\Si=\Si_t$, $t<0$,  
\beaa
\int_\Si\big(|\nab k|^2+3(k^2)^{mn}(E_{mn}+(k^2)_{mn})-\f12 |k|^2 |k|^2\big)=
\int_\Si|H|^2
\eeaa
Interpreting $k$ as $3\times 3$ symetric matrices we can write 
\beaa
3(k^2)^{mn}(k^2)_{mn}-\f12 |k|^2 |k|^2=3 \tr(k^4)-\f12 (\tr k^2)^2
\eeaa
Observe that we have the pointwise inequality\footnote{Indeed diagonalizing $k$ it suffices to prove the
inequality for arbitrary  real numbers $a,b,c$,
$3(a^2+b^2+c^2)^2\ge a^4+b^4+c^4$.}, for an arbitrary symmetric matrix $k$,
$ \tr (k^4)\ge \frac{1}{3}|k|^4$. Therefore,
\beaa
\int_\Si|\nab k|^2+\f12 |k|^4&\le& \int_\Si |H|^2+\int_\Si  |E|\, |k|^2
\le \int_\Si |H|^2 +|E|^2+\int_\Si\frac{1}{4}|k|^4 
\eeaa
This  proves the following:
\begin{proposition}
\label{prop:L2-estim-k}On  any leaf $\Si$ of a constant mean curvature foliation $\Si_t$
the second fundamental form $k$ verifies the estimate,
\be{eq:L2-estim-k}
\int_\Si|\nab k|^2+\frac{1}{4} |k|^4\le \int_\Si |H|^2 +|E|^2=\int_\Si |\R|^2.
\end{equation}
\end{proposition}
In view of the energy estimate of   proposition \ref{prop:L2-curv}
we  derive,
\begin{corollary}\label{co:firstder.k}
The following estimates hold true with a constant $C$ depending only on
$\De_0$ and $t_*$,
\bea
\|\nab k(t)\|_{L^2}+\|k(t)\|_{L^4}\le C \RR_{0}
\eea
\end{corollary}
\subsection{Higher derivatives estimates  for $k$}
To derive second  derivative estimates for $k$ we rewrite the curl equation in 
\eqref{eq:Hodge-k} in the form,
\beaa
\nab_i k_{jm}-\nab_j k_{im}=\in_{ij^s} H_{sm}
\eeaa
Differentiating we  obtain,
\beaa
\nab^i \nab_i k_{jm} -\nab^i \nab k_{im} =\in_{ij^s}\nab^i  H_{sm}
\eeaa
or, symbolically,
\beaa
\lap k= R \star  k + \nab H
\eeaa
where $R \star  k$ is a quadratic expression with respect to the Ricci
curvature $R$ of $\Si_t$ and $k$. Thus, since, the Ricci curvature $R$
can be expressed in  the form,
\beaa
R_{ij} -k_{ia} k^a_{\,\, j} +\tr k k_{ij} =E_{ij}
\eeaa
we derive,
\beaa
|\lap k| \le |k|^3  +|E||k|+ |\nab H|
\eeaa
Therefore,
\beaa
  \int_{\Si_t}|\lap k|^2& \le&  \int_{\Si_t}\big( |k|^6 + |E|^2 |k|^2+ |\nab H|^2\big)
  \eeaa
It is easy to see by a standard integration by parts argument that,
\beaa
\int_{\Si_t} |\nab^2 k(t)|^2\le  \int _{\Si_t}  |\lap k|^2+\int_{\Si_t}| R|^2|\nab k|^2
\eeaa
Consequently,
\beaa
  \int_{\Si_t}|\nab^2 k |^2& \le&  \int_{\Si_t}\big( |k|^6 + |E|^2 |k|^2+ |\nab H|^2 \big)+ \int_{\Si_t} |R|^2 |\nab k|^2
  \eeaa
Therefore, since $\|k\|_{L^\infty}\le \De_0$,
\bea
\|\nab^2 k(t)\|_{L^2}
&\le &\De_0^2\, \big(\|k(t)\|_{L^4}^4+\|\R\|^2_{L^2}\big)+\|\nab H(t)\|_{L^2}^2\\
\nn&+&\|R(t)\|_{L^\infty}\|\nab k(t)\|_{L^2}^2\label{last-estimate}
\eea
It is easy to see that 
$
\|\nab H(t)\|_{L^2}^2\les \|\D\R(t)\|_{L^2}$. 
Also, 
\beaa
\|R(t)\|_{L^\infty}&\le&\ |\R\|_{L^\infty}+\|k\|_{L^\infty}^2
\les \|\R(t)\|_{H^2}+\De_0^2
\eeaa
Therefore, in view of theorem \ref{mainthm-energy}
and  the bounds for  $\|\nab k(t)\|_{L^2}$ and $\|k(t)\|_{L^4}$ of  corollary \ref{co:firstder.k} we derive from  \eqref{last-estimate},
\beaa
\|\nab^2 k(t)\|_{L^2}& \le& C
\eeaa
with $C$ a constant depending only on $\De_0, t_*$ and $\RR_0$.

Differentiating once more the equation for $\lap k$ and proceeding in the same fashion we can also derive  similar bounds for the third derivatives of $k$. This proves the following.
\begin{theorem}
\label{thm-last}
The second fundamental form $k$ of the $t$ foliation  satisfies the following
estimate, for all $t_0\le t<\t_*$,
\bea
\|\nab^3 k(t)\|_{L^2}+\|\nab^2 k(t)\|_{L^2}+\|\nab k(t)\|_{L^2}+\|k(t)\|_{L^4}\le C
\eea
with $C$ a constant depending only on $\De_0, t_*$ and $\RR_0$.
\end{theorem}

\end{document}